\documentclass[aap,MSNbibl,citesort,dvips]{arximspdf}
\usepackage{mathbh}
\usepackage{graphicx}

\doi{10.1214/11-AAP789}
\volume{22}
\issue{3}
\pubyear{2012}
\firstpage{1266}
\lastpage{1300}

\makeatletter

\newcommand{\Lemmaref}[1]{\mbox{\fontsize{8.36pt}{10.36pt}\selectfont{\ref{#1}}}}
\newcommand{\Lemmarref}[1]{\mbox{\fontsize{6.6pt}{8.36pt}\selectfont{\ref{#1}}}}

\newtheorem{theorem}{Theorem}[section]
\newtheorem{lemma}[theorem]{Lemma}
\newtheorem{proposition}[theorem]{Proposition}

\newproclaim{definition}[theorem]{Definition}
\newproclaim{remark}[theorem]{Remark}

\newcommand{\deli}{\del_i}
\newcommand{\delj}{\del_j}
\newcommand{\delij}{\del_{ij}}
\newcommand{\delii}{\del_{ii}}

\newcommand{\Span}{\operatorname{Span}}
\newcommand{\tr}{\operatorname{tr}}
\newcommand{\Var}{\operatorname{Var}}

\newcommand{\Ree}{\operatorname{Re}}
\newcommand{\Imm}{\operatorname{Im}}

\newcommand{\of}{\circ}
\newcommand{\bb}{\mathbb}

\newcommand{\Xijk}{X_{i,j}^{(k)}}
\newcommand{\Yijk}{Y_{i,j}^{(k)}}
\newcommand{\Irho}{\bb I_\rho}
\newcommand{\ci}{\sqrt{-1}}

\newcommand{\del}{\partial}
\newcommand{\dist}{\operatorname{dist}}
\newcommand{\tilai}{\tilde{a}_i}
\newcommand{\aibdd}{|a_i|< n^{1/2-\alpha/4}}
\newcommand{\tilX}{\tilde{X}}
\newcommand{\poslog}{\log_+}
\newcommand{\neglog}{\log_-}
\newcommand{\sumiNodd}{\mathop{\sum_{i=1}^{N-1}}_{i\ \mathrm{odd}}}

\makeatother

\begin{document}
\begin{frontmatter}

\title{Universality and the circular law for sparse random matrices}
\runtitle{Universality for sparse random matrices}

\begin{aug}
\author[A]{\fnms{Philip Matchett} \snm{Wood}\corref{}\thanksref{t1}\ead[label=e1]{pmwood@math.stanford.edu}}
\runauthor{P. M. Wood}
\affiliation{Stanford University}
\address[A]{Department of Mathematics\\
Stanford University\\
Building 380\\
Stanford, California 94305 \\
USA\\
\printead{e1}} 
\end{aug}

\thankstext{t1}{Supported by an NSF Postdoctoral Scholar grant.}

\received{\smonth{9} \syear{2010}}
\revised{\smonth{4} \syear{2011}}

%
\begin{abstract}
The universality phenomenon asserts that the distribution of the
eigenvalues of random matrix with i.i.d. zero mean, unit variance
entries does not depend on the underlying structure of the random
entries.
For example, a plot of the eigenvalues of a random sign matrix, where
each entry is $+1$ or $-1$ with equal probability, looks the same as an
analogous plot of the eigenvalues of a random matrix where each entry
is complex Gaussian with zero mean and unit variance. In the current
paper, we prove a universality result for sparse random $n$ by $n$
matrices where each entry is nonzero with probability $1/n^{1-\alpha}$
where $0 <\alpha\le1$ is any constant.
One consequence of the sparse universality principle is that the
circular law holds for sparse random matrices so long as the entries
have zero mean and unit variance, which is the most general result for
sparse random matrices to date.
\end{abstract}

%
\begin{keyword}[class=AMS]
\kwd{15A52}.
\end{keyword}
\begin{keyword}
\kwd{Random matrix}
\kwd{sparse matrix}
\kwd{circular law}.
\end{keyword}

\end{frontmatter}

\section{Introduction}\label{intro}

Given an $n$ by $n$ complex matrix $A$, we define the \textit{empirical
spectral distribution} (which we will abbreviate \textit{ESD}), to be the
following discrete probability measure on $\bb C$:
\[
\mu_{A}(z):= \frac
1n|\{ 1 \le i \le n\dvtx\Ree(\lambda_i) \le\Ree(z) \mbox{ and }
\Imm(\lambda_i) \le\Imm(z)\} |,
\]
where $\lambda_1,\lambda_2,\ldots,\lambda_n$
are the eigenvalues of $A$ with multiplicity. In this paper, we focus
on the
case where $A$ is chosen from a probability distribution on $\bb
M_n(\bb C)$,
the set of all $n$ by $n$ complex matrices, and thus $\mu_A$ is a randomly
generated discrete probability measure on $\bb C$.

\subsection{Background: Universality and the circular law}

Suppose that $A_n$ is an $n$ by $n$ matrix with i.i.d. random entries, each
having zero mean and unit variance.\vadjust{\goodbreak} The distribution of the eigenvalues of
$(1/\sqrt n)A_n$ approaches the uniform distribution on the unit disk
as $n$
goes to infinity, a phenomenon known as the circular law. The nonsparse
circular law has been proven in many special cases by many authors, including
Mehta~\cite{Mehta1967} (Gaussian case),
Girko~\cite{Girko1984,Girko2004},
Edelman~\cite{Edelman1988} (real Gaussian case),
Bai~\cite{Bai1997} and Bai and Silverstein~\cite{BSilverstein2010} [continuous
case with bounded $(2+\delta)$th moment, for $\delta>0$],
G\"otze and Tikhomirov~\cite{GTikhomirov2007} (sub-Gaussian case)
and~\cite{GTikhomirov2010} [bounded $(2+\delta)$th moment, for
$\delta>0$],
Pan and Zhao~\cite{PZhou2010} (bounded $4$th moment)
and Tao and Vu~\cite{TVu2008} [bounded $(2+\delta)$th moment, for
$\delta>0$].
The following, due to Tao and Vu~\cite{TVuK2008},
Theorem 1.10, is the current
best result, requiring only zero mean and unit variance (see also~\cite{TVu2009LOuniv}).
\begin{theorem}[(Nonsparse circular law (\cite{TVuK2008}, Theorem 1.10))]
\label{TVuCircLaw} Let $X_n$ be the $n$ by
$n$ random matrix whose entries are i.i.d. complex random variables
with mean
zero and variance one. Then the ESD of $\frac1{\sqrt n} X_n$ converges (both
in probability and in the almost sure sense) to the uniform
distribution on
the unit disk.
\end{theorem}

There has also been recent interest in generalizations of the circular
law to
random matrix ensembles where finite variance is relaxed (see
\cite{BCChafai2010})
and where some dependence among the entries is allowed (see
\cite{CDSly2010,BCChafai2010a}).

Proving convergence in the almost sure sense is, in general, harder
than proving
convergence in probability, and in the current paper, we will focus
exclusively on convergence in probability. See Section~\ref{SSBackDefs}
toward the end of the \hyperref[intro]{Introduction} for a description of convergence in
probability and in the almost sure sense for the current context.

In~\cite{TVuK2008}, Tao and Vu ask the following natural question:
what analog
of Theorem~\ref{TVuCircLaw} is possible in the case where the matrix is
sparse, where entries become more likely to be zero as $n$ increases, instead
of entries having the same distribution for all $n$? One goal of the
current paper is to provide an answer to this question in the form of
Theorem~\ref{SpCirc} (see below), which proves the circular law for sparse
random matrices with i.i.d. entries. In Figure~\ref{FCirc}, parts (b)
and (d)
give examples of the nonsparse circular law for Bernoulli and Gaussian random
variables, and parts (a) and (c) give examples of the sparse circular
law for
Bernoulli and Gaussian random variables.\looseness=1

%
%
\begin{figure}
\begin{tabular}{cc}

\includegraphics{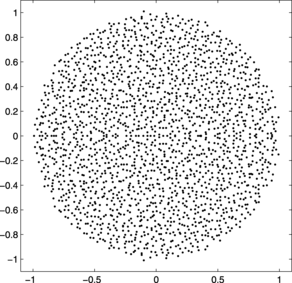}
 & \includegraphics{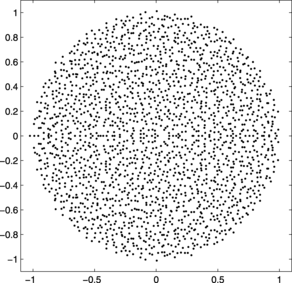}\\
(a) & (b)\\[4pt]

\includegraphics{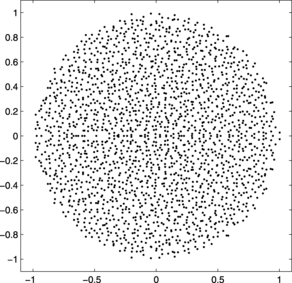}
 & \includegraphics{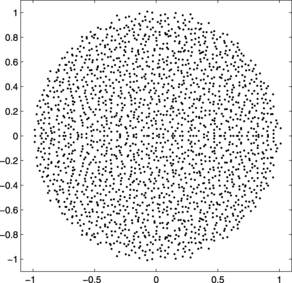}\\
(c) & (d)
\end{tabular}
\caption{The four figures above illustrate that the circular law holds
for
Bernoulli and Gaussian random matrix ensembles in both the sparse and
nonsparse cases. Each plot is of the eigenvalues of a $2\mbox{,}000$ by $2\mbox{,}000$
random matrix with i.i.d. entries. In the first column
[parts \textup{(a)} and \textup{(c)}]
the matrices are sparse with parameter $\alpha=0.4$, which means each
entry is
zero\vspace*{1pt} with probability $1-\frac1{n^{0.6}}$, and in the second column [parts
\textup{(b)} and \textup{(d)}] the matrices are not sparse (i.e., $\alpha= 1$). In the first
row, both matrix ensembles are Bernoulli, so each nonzero entry is equally
likely to be $-1$ or $1$, and in the second row, the ensembles are Gaussian,
so the nonzero entries are drawn from a Gaussian distribution with
mean zero
and variance one.} \label{FCirc}
\end{figure}

The mathematical literature studying the eigenvalues of sparse random matrices
is distinctly smaller than that for nonsparse random matrices (there are,
however, some nonrigorous approaches from a physics perspective,\break
e.g.,~\cite{FSKhoruzhenko1998}). Most authors in mathematics and
physics have focused on studying the eigenvalues in the symmetric case,
including
\cite
{SKobayashi1976,RodgersBray1988,MFyodorov1991,DSSulkowski1997,SCugliandolo2002,NTanaka2007,RPKTakeda2008,Kuhn2008,Sodin2009}.
There has been, however, some recent and notable progress for nonsymmetric
sparse random matrices. G\"otze and Tikhomirov
\cite{GTikhomirov2007,GTikhomirov2010} provide sparse versions for their
proofs of the circular law with some extra conditions.
In~\cite{GTikhomirov2007} they use the additional assumptions that the entries
are sub-Gaussian and that each entry is zero with probability $\rho_n$ where
$\rho_n n^4 \to\infty$ as $n\to\infty$, and in~\cite{GTikhomirov2010}
they use the additional assumption that the entries have bounded
$(2+\delta)$th moment.
The strongest result in the literature for nonsymmetric sparse random
matrices is due to Tao and Vu~\cite{TVu2008} who in 2008 proved a~sparse
version of the circular law with the assumption of bounded $(2+\delta)$th
moment (note that~\cite{TVu2008} proves almost sure convergence,
rather than
convergence in probability as shown by~\cite{GTikhomirov2007,GTikhomirov2010}).
%
\begin{theorem}[(\cite{TVu2008}, Theorem 1.3)]
Let $\alpha>0$ and $\delta>0$ be arbitrary positive constants. Assume that
$x$ is a complex random variable with zero mean and finite $(2 + \delta)$th
moment. Set $\rho=n^{-1+\alpha}$ and let $A_n$ be the matrix with
each entry
an i.i.d. copy of $\frac1{\sqrt\rho}\Irho x$, where $\Irho$ is a\vspace*{1pt} random
variable independent of $x$, and $\Irho$ takes the value $1$ with probability
$\rho$ and the value $0$ with probability $1-\rho$. Let
$\mu_{(1/({\sigma\sqrt{ n}}))A_n}$ be the ESD of $\frac1{\sigma
\sqrt
{n}}A_n$, where $\sigma^2$ is, as usual, the variance of $x$. Then
$\mu_{(1/{\sqrt{\sigma n}})A_n}$ converges in the almost sure
sense to the
uniform distribution $\mu_\infty$ over the unit disk as $n$ tends to
infinity.\vspace*{-1pt}
\end{theorem}

In this paper, we prove a sparse circular law without the bounded
$(2+\delta)$th moment condition, with our work being motivated by the
proof in
\cite{TVuK2008} of the (nonsparse) circular law in the general zero mean,
unit variance case.

There has been much recent interest in demonstrating universal behavior for
the eigenvalues of various types of random matrices. The following
theorem is
a fundamental result from~\cite{TVuK2008}.
For a matrix $A = (a_{ij})_{1\le i,j\le n}$, we will use $\| A
\|_2$ to
denote the Hilbert--Schmidt norm, which is defined by $ \| A
\|_2 =
\operatorname{trace}{A A^{*}} = (\sum_{1\le i,j\le n}
|a_{ij}|^2
)^{1/2}$.\vspace*{-1pt}
\begin{theorem}[(Universality principle~\cite{TVuK2008})]\label{TVuKthm}
Let $x$ and $y$ be complex random variables with zero mean and unit variance.
Let $X_n := (x_{ij})_{1\le i,j\le n}$ and $Y_n := (y_{ij})_{1\le i,j\le
n}$ be
$n\times n$ random matrices whose entries $x_{ij},y_{ij}$ are i.i.d.
copies of
$x$ and $y$, respectively. For each $n$, let $M_n$ be a deterministic
$n\times n$ matrix satisfying
%
%
\begin{equation}\label{eq3} \sup_n
\frac1{n^2}\| M_n \|_2^2 < \infty.
\end{equation}
Let $A_n:= M_n + X_n$ and $B_n:= M_n + Y_n$. Then $\mu_{(1/{\sqrt n})
A_n} - \mu_{(1/{\sqrt n}) B_n}$ converges in probability to zero.\vspace*{-1pt}
\end{theorem}

The universality\vspace*{-1pt} principle as proven in~\cite{TVuK2008}, Theorem 1.5, also
includes an additional hypothesis under which $\mu_{(1/{\sqrt n})
A_n} -
\mu_{(1/{\sqrt n}) B_n}$ converges\vspace*{2pt} almost surely to zero (see
\cite{TVuK2008} for details). In~\cite{TVuK2008}, Tao and Vu suggest the
project of extending their universality principle for random matrices
to the case of sparse random matrices. In this paper, we will follow the
program developed in~\cite{TVuK2008} and prove a universality
principle for
sparse random matrices.\vspace*{-1pt}

\subsection{New results for sparse random matrices}

We begin by defining the type of sparse matrix ensemble that we will consider
in this paper.\vspace*{-1pt}
\begin{definition}[(Sparse matrix ensemble)]
Let $0 < \alpha\le1$ be a constant, and let $\Irho$ be the random variable
taking the value $1$ with probability $\rho:=n^{-1+\alpha}$ and the
value $0$
with probability $1-\rho$. Let $x$ be a complex random variable that is
independent of $\Irho$. The \textit{$n$ by $n$ sparse matrix ensemble
for~$x$
with parameter $\alpha$} is defined to be the matrix $X_n$ where each
entry is
an i.i.d. copy of $\frac1{\sqrt\rho}\Irho x$.\vadjust{\goodbreak}
\end{definition}

The main result of the current paper is the following:
\begin{theorem}[(Sparse universality principle)]\label{MainThm}
Let $0< \alpha\le1$ be a constant, and let $x$ be a random variable with
mean zero and variance one. Let~$X_n$ be the $n$ by $n$ sparse matrix
ensemble for $x$ with parameter $\alpha$, and let $Y_n$ be the $n$ by $n$
matrix having i.i.d. copies of $x$ for each entry (in particular,~$Y_n$ is
\textup{not} sparse). For each $n$, let $M_n$ be a deterministic $n$ by $n$
matrix such that
%
%
\begin{equation}\label{Mnbdd}
\sup_n \frac1{n^2} \| M_n \|^2_2 < \infty,
\end{equation}
and let $A_n:= M_n + X_n$ and $B_n := M_n + Y_n$.
Then, ${\mu_{(1/{\sqrt n}) A_n}} - \mu_{(1/{\sqrt n}) B_n}$
converges in
probability to zero.
\end{theorem}

Figure~\ref{FUniv} gives an illustration of Theorem~\ref{MainThm} with
nontrivial $M_n$ for sparse and nonsparse Bernoulli and Gaussian ensembles.
%
%
\begin{figure}
\begin{tabular}{@{}cc@{}}

\includegraphics{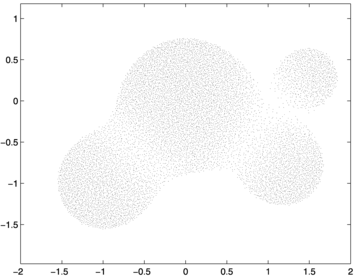}
 & \includegraphics{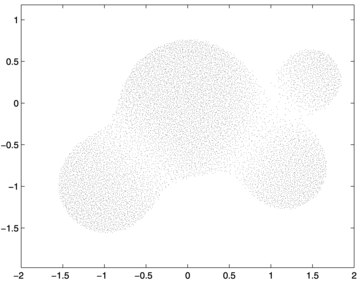}\\
(a) & (b)\\[4pt]

\includegraphics{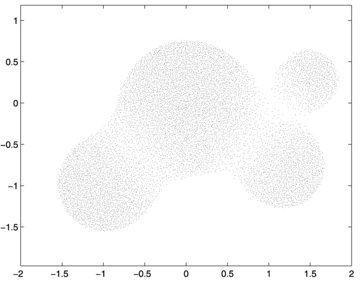}
 & \includegraphics{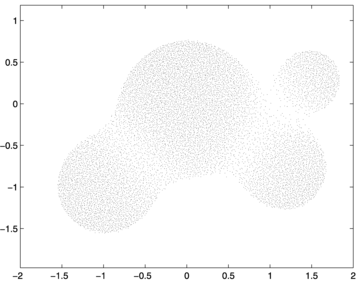}\\
(c) & (d)
\end{tabular}
\caption{The four plots above illustrate that the universality
principle holds
for Bernoulli and Gaussian random matrix ensembles in both the sparse and
nonsparse cases. Each plot is of the eigenvalues of a $10\mbox{,}000$ by $10\mbox{,}000$
random matrix with of the form $M_n + X_n$, where $M_n$ is a fixed, nonrandom
matrix, and $X_n$ contains i.i.d. entries. For each\vspace*{-1pt} of the four plots,
$\frac{1}{\sqrt{n}}M_n$ is the diagonal matrix with the first
${\lfloor n/4 \rfloor}$
diagonal entries\vspace*{1pt} equal to ${-1-\ci}$, the next $\lfloor n/6
\rfloor$ diagonal entries
equal to $1.2-0.8\ci$, the next $n/12$ diagonal entries equal to
$1.5+0.3\ci$
and the remaining entries equal to zero. In the first column [parts
\textup{(a)} and~\textup{(c)}] the matrices $X_n$ are sparse with parameter $\alpha=0.5$, which means
each entry is zero with probability $1-\frac1{n^{0.5}}$, and in the second
column\vspace*{1pt} [parts \textup{(b)} and \textup{(d)}] the matrices $X_n$ are not sparse (i.e.,
$\alpha= 1$). In the first row, both matrix ensembles are Bernoulli,
so each
nonzero entry of $X_n$ is equally likely to be $-1$ or $1$, and in the second
row, the ensembles are Gaussian, so the nonzero entries of $X_n$ are drawn
from a Gaussian distribution with mean zero and variance one.}
\label{FUniv}\vspace*{-6pt}
\end{figure}
In~\cite{Rogers2010}, a method is given for predicting the eigenvalue
distributions of a random matrix plus a deterministic matrix and also
of a
random matrix multiplied by a deterministic matrix.

Relating the sparse case to the nonsparse case in the above theorem is quite
useful, since many results are known for random matrices with
nonsparse i.i.d. entries, including a number of results in \cite
{TVuK2008}. One
of the motivating consequences of Theorem~\ref{MainThm} is the following
result, which is a~combination of Theorems~\ref{MainThm}~and~\ref{TVuCircLaw},
the nonsparse circular law proven in~\cite{TVuK2008}.

\begin{theorem}[(Sparse circular law)]\label{SpCirc}
Let $0 < \alpha\le1$ be a constant, and let $x$ be a random complex variable
with mean zero and variance one. Let~$X_n$ be the sparse matrix
ensemble for
$x$ with parameter $\alpha$. Then the ESD for $\frac1{\sqrt n} X_n$ converges
in probability to the uniform distribution on the unit disk.
\end{theorem}


An illustration of Theorem~\ref{SpCirc} appears in Figure~\ref{FCirc}.
Note that the sparse circular law (Theorem~\ref{SpCirc}) does not hold when
$\alpha=0$, since the probability of a row of all zeroes approaches a constant
as $n\to\infty$, and thus with probability tending to 1 as $n\to
\infty$, a
constant fraction of the rows contain all zeroes. Reasoning in analogy with
the Hermitian case, where Wigner's semicircle law holds so long as
$n\rho
\to\infty$ (see~\cite{Wigner1958}), it seems\vspace*{1pt} possible that one might
be able to
prove the circular law in the case where $\rho= \frac{\log n}{n}$
(see~\cite{RCastillo2009} for further evidence). One might
also consider analogs of other models of sparseness that have been used
in the
Hermitian case; for example, see~\cite{TVWang2010,DPal2010}.


In the nonsparse case, Tao, Vu and Krishnapur~\cite{TVuK2008} also
give a number of extensions and
generalizations, one of which is the circular law for shifted
matrices, including the case where the entries of a random matrix have
constant, nonzero mean.
\begin{theorem}[(Nonsparse circular law for shifted matrices (\cite
{TVuK2008}, Corollary~1.12))]
\label{TShift}
Let $X_n$ be the $n$ by $n$ random matrix whose entries are i.i.d.
complex random
variables with mean $0$ and variance $1$, and let $M_n$ be a~deterministic
matrix with rank $o(n)$ and obeying inequality (\ref{eq3}). Let $A_n:=M_n+X_n$.
Then the ESD of $\frac1{\sqrt n} A_n$ converges (both in probability
and in
the almost sure sense) to the uniform distribution on the unit
disk.\vadjust{\goodbreak}
\end{theorem}

Because Theorem~\ref{TShift} applies to nonsparse matrices of the form
\mbox{$M_n+X_n$}, it can be directly combined with the sparse universality principle
of Theorem~\ref{MainThm} to yield the following result:\vspace*{-2pt}
\begin{theorem}[(Sparse circular law for shifted matrices)]\label{TSpShift}
Let $0<\alpha\le1$ be a constant, and let $x$ be a complex random variable
with mean $0$ and variance~$1$.
Let $X_n$ be the $n$ by $n$ sparse random matrix ensemble with
parameter~$\alpha$, let $M_n$ be a deterministic matrix with rank $o(n)$ and obeying
inequality~(\ref{eq3}) and let \mbox{$A_n:=M_n+X_n$}.
Then the ESD of $\frac1{\sqrt n} A_n$ converges in probability to the uniform
distribution on the unit disk.\vspace*{-2pt}
\end{theorem}

An example of Theorems~\ref{TShift} and~\ref{TSpShift} appears in
Figure~\ref{FShift}.

%
%
\begin{figure}
\begin{tabular}{@{}cc@{}}

\includegraphics{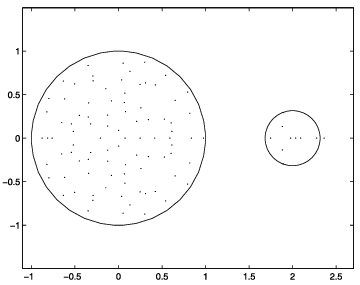}
 & \includegraphics{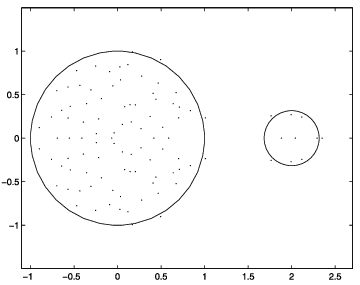}\\
(a) & (b)\\[4pt]

\includegraphics{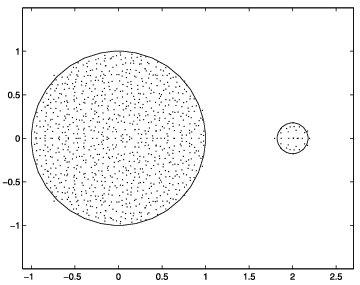}
 & \includegraphics{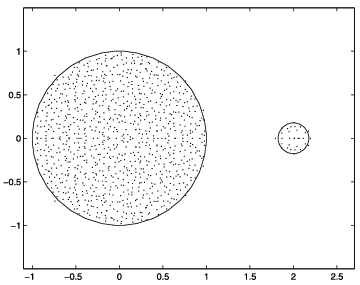}\\
(c) & (d)
\end{tabular}
\caption{These six figures illustrate that the circular law holds
for
shifted sparse Bernoulli and shifted nonsparse Bernoulli random matrix
ensembles. Each plot is of the eigenvalues of an $n$ by~$n$ (with $n$ as
specified) random matrix of the form $M_n + X_n$, where $M_n$ is a nonrandom
diagonal matrix with the first $\lfloor\sqrt n \rfloor$
diagonal entries
equal to
$2\sqrt n$ and the remaining entries equal to zero, and $X_n$ contains i.i.d.
random entries. In the first column [parts \textup{(a)} $n=100$,
\textup{(c)} $n=1\mbox{,}000$ and \textup{(e)} $n=10\mbox{,}000$] the matrices
are sparse with parameter $\alpha=0.4$, which means each entry is zero with
probability $1-\frac1{n^{0.6}}$, and in the second column [parts
\textup{(b)} $n=100$, \textup{(d)} $n=1\mbox{,}000$
and \textup{(f)} $n=10\mbox{,}000$] the matrices are not sparse (i.e., $\alpha= 1$). The matrix
ensembles are Bernoulli, so each nonzero entry is equally likely to be $-1$
or $1$. As $n$ increases, the ESDs in both the sparse and nonsparse cases
approach the uniform distribution on the unit disk. Empirically, the small
circle on the right, which has roughly $\sqrt n$ eigenvalues in and
near it,
shrinks until its contribution to the ESD is negligible (as drawn, the small
circle has radius $n^{-1/4}$).}\vspace*{-3pt}
\label{FShift}
\end{figure}
\setcounter{figure}{2}
\begin{figure}
\begin{tabular}{@{}cc@{}}

\includegraphics{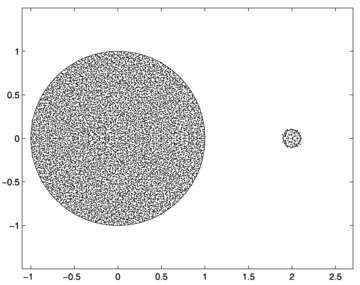}
 & \includegraphics{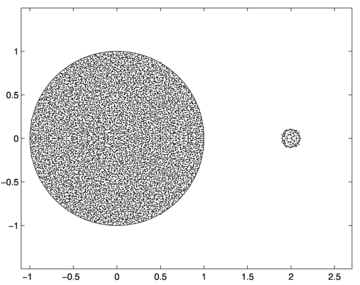}\\
(e) & (f)\vspace*{-4pt}
\end{tabular}
\caption{(Continued).}\vspace*{-3pt}
\end{figure}

The simple lemma below is an essential component for adapting arguments from
\cite{TVuK2008} to the sparse case, and illustrates a critical
transition that
occurs when $\alpha= 0$.\vspace*{-2pt}
\begin{lemma}\label{SparseDom}
Let $\xi$ be a complex random variable such that $\bb E|\xi|<
\infty$.
Let $X$ be a sparse version of $\xi$, namely $X:= \Irho\xi/ \rho$, where
$\rho= n^{-1+\alpha}$, where $0< \alpha\le1$ is a constant. Then
\[
\bb E\bigl(\bigl|\mathbh{1}_{\{X>n^{1-\alpha/2}\}} X
\bigr|\bigr) \to0
\]
as $n \to\infty$.\vspace*{-2pt}
\end{lemma}
\begin{pf}
The key steps to this proof are using independence of $\Irho$ and $\xi
$, and
applying monotone convergence. We compute
\begin{eqnarray*}
\bb E\bigl(\bigl|\mathbh{1}_{\{|X|>n^{1-\alpha
/2}\}} X\bigr|\bigr) &=&
\bb E\bigl(\bigl|\mathbh{1}_{\{|\Irho\xi| >n^{\alpha
/2}\}} \Irho\xi/ \rho\bigr|\bigr) \\
&\le&\frac1\rho\bb E\bigl(\bigl|\mathbh{1}_{\{|\xi|
>n^{\alpha/2}\}} \Irho\xi\bigr|\bigr) \\
&=&\bb E\bigl(\bigl|\mathbh{1}_{\{|\xi| >n^{\alpha
/2}\}} \xi\bigr|\bigr). 
\end{eqnarray*}
Finally, $\bb E(|\mathbh{1}_{\{\xi>n^{\alpha/2}\}
} \xi|) \to0$ as $n\to
\infty$ by
monotone convergence.\vspace*{-2pt}
\end{pf}
\begin{remark}\label{Rdecalpha}
The proof of Lemma~\ref{SparseDom} illustrates that $\rho=1/n$ is a~transition
point for sparse random variables of the type $\Irho\xi$ where the arguments
for universality break down. Notably, the proof of Lemma~\ref{SparseDom}
also works for $\alpha$ depending on $n$ so long as $\alpha\log n$
tends to
infinity as $n\to\infty$; for example, $\alpha= \frac1{\log\log n}$ is
suitable. It would be interesting to see if the universality principle
extends to parameters $\alpha$ that tend slowly to zero as $n\to
\infty$.\vspace*{-2pt}
\end{remark}

\subsection{Further directions}

There are a number of natural further directions to consider with
respect to
the sparse universality principle Theorem~\ref{MainThm}.
One natural question is whether Theorem~\ref{MainThm} can be generalized
to prove almost sure convergence in addition to proving convergence in
probability. A result\vadjust{\goodbreak} of Dozier and Silverstein \cite
{DSilverstein2007} is
one of the ingredients used in~\cite{TVuK2008} to prove almost sure
convergence; however, there does not seem to be a~sparse analog of
\cite{DSilverstein2007}. Proving a sparse analog of~\cite{DSilverstein2007}
would be a~substantial step toward proving a universality principle with
almost sure convergence (see Remark~\ref{RDSSparse}), though there may be
other avenues as well. Finally, a~general question of interest would be to
study the rates of convergence for the universality principle. Convergence
seems reasonably fast in the nonsparse case; however, empirical evidence
indicates that convergence is slower in the sparse case and may in fact depend
on the underlying type of random variables; see Figure~\ref{FRate} for an
%
%
%
\begin{figure}
\begin{tabular}{@{}cc@{}}

\includegraphics{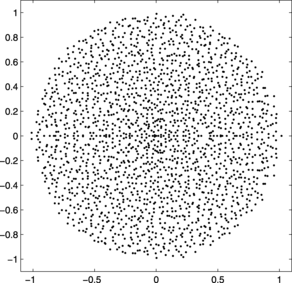}
 & \includegraphics{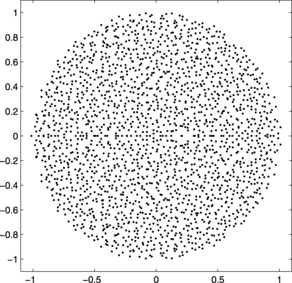}\\
(a) & (b)\\[4pt]

\includegraphics{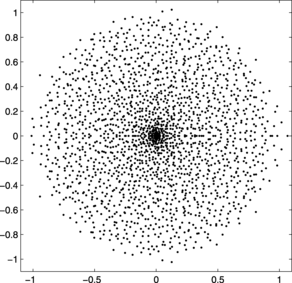}
 & \includegraphics{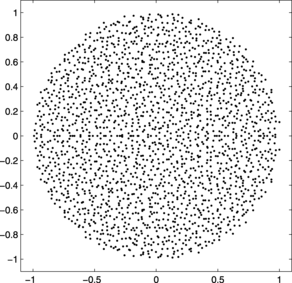}\\
(c) & (d)\vspace*{-2pt}
\end{tabular}
\caption{The four figures above indicate that the rates of convergence
to the
uniform distribution on the unit disk for sparse Bernoulli and sparse
Gaussian random matrix ensembles are apparently not the same as each other,
and that in particular the sparse Gaussian case converges more slowly
that the
nonsparse case. Each plot is of the eigenvalues of a $2\mbox{,}000$ by $2\mbox{,}000$ random
matrix with i.i.d. entries. In the first column
[parts \textup{(a)} and \textup{(c)}] the
matrices are sparse with parameter $\alpha=0.2$, which means each
entry is
zero with probability $1-\frac1{n^{0.8}}$, and in the second\vspace*{1pt} column
[parts \textup{(b)} and \textup{(d)}] the matrices are not sparse (i.e., $\alpha= 1$). In the first
row, both matrix ensembles are Bernoulli, so each nonzero entry is equally
likely to be $-1$ or $1$, and in the second row, the ensembles are Gaussian,
so the nonzero entries are drawn from a Gaussian distribution with
mean zero
and variance one.}
\label{FRate}\vspace*{-3pt}
\end{figure}
example. A bound on convergence rates in the nonsparse case where the
$(2+\delta)$th moment is bounded is given in~\cite{TVu2008}, Section~14.\vspace*{-2pt}

\subsection{Definitions of convergence and notation}\label{SSBackDefs}

Let $X$ be a random variable taking values in a Hausdorff topological space.
We say that $X_n$ \textit{converges in probability} to $X$ if for every
neighborhood $N_X$ of $X$, we have
\[
\lim_{n\to\infty} \Pr(X_n \in N_X) = 1.
\]
Furthermore, we say that $X_n$ \textit{converges almost surely} to $X$ if
\[
\Pr\Bigl(\lim_{n\to\infty} X_n = X\Bigr) = 1.
\]

If $C_n$ is a sequence of random variables taking values in $\bb R$, we
say that~$C_n$ is \textit{bounded in probability} if
\[
\lim_{K\to\infty} \liminf_{n\to\infty} \Pr(C_n \le K) = 1.
\]

In the current paper, we are interested in how a randomly generated
sequence of
ESDs $\mu_{A_n}$ converges as $n \to\infty$, and so we will put the standard
vague topology on the space of probability measures on $\bb C$.
In particular, if~$\mu_n$ and $\mu_n'$ are randomly generated
sequences of
measures on $\bb C$, then $\mu_n-\mu_n'$ converges in probability to zero
if for every smooth function with compact support $f$ and for every
$\varepsilon>0$, we have
\[
\lim_{n\to\infty} \Pr\biggl(\biggl|\int_{\bb C} f \,d\mu_n - \int_{\bb
C} f \,d\mu_n'\biggr| \le\varepsilon\biggr) = 1.
\]
Furthermore, $\mu_n -\mu_n'$ converges to zero almost surely if for every
smooth function with compact support $f$ and for every $\varepsilon>0$, the
expression $|\int_{\bb C} f \,d\mu_n - \int_{\bb C} f \,d\mu
_n'| $
converges to $0$ with probability $1$.\vadjust{\goodbreak}

For functions $f$ and $g$ depending on $n$, we will make use of the asymptotic
notation $f = O(g)$ to mean that there exists a positive constant $c$
(independent of~$n$) such that $f\le cg$ for all sufficiently large $n$.
Also, we will use the asymptotic notation $f = o(g)$ to mean that $f/g
\to0$
as $n\to\infty$.\vspace*{-2pt}

\subsection{Paper outline}

Recall that the sparseness is determined by $\rho:= n^{-1+\alpha}$.
In the
remaining sections, we will follow the approach used in \cite
{TVuK2008} to
prove a universality principle for sparse random matrices when $\alpha
> 0$.
In Section~\ref{SMainPf},\vadjust{\goodbreak} we outline the main steps of the proof,
highlighting a general result about convergence of ESDs from
\cite{TVuK2008}
that essentially reduces the question of convergences of ESDs to a
question of
convergence of the determinants of the corresponding matrices (one of
which is
sparse, and the other of which is not). Section~\ref{SPropPf} gives a proof
of a sparse version of the necessary result on convergence of determinants
based on a least singular value bound for sparse matrices in~\cite{TVu2008}
and two lemmas, which are proved in Sections~\ref{SLemHdimPf}
and~\ref{SLemLdimPf}, respectively. In Section~\ref{SLemLdimPf}, we make
use of
a complex version of a result of Chatterjee~\cite{Chatterjee2008} (namely,
Theorem~\ref{ComplexChat2008Thm}) which requires adapting Krishnapur's ideas
in~\cite{TVuK2008}, Appendix C, to a sparse context
(\cite{TVuK2008}, Appendix C, is dedicated to proving a universality principle
for nonsparse random matrices where the entries are not necessarily i.i.d.).

\section{\texorpdfstring{Proof of Theorem \protect\ref{MainThm}}{Proof of Theorem 1.5}}\label{SMainPf}

The following result was proven by Tao and Vu
\cite{TVuK2008}, Theorem 2.1, and can be applied directly in proving
Theorem~\ref{MainThm}.
All logarithms in this paper are natural unless otherwise noted.
\begin{theorem}[(\cite{TVuK2008})] \label{TVuK21} Suppose for each $n$
that $A_n, B_n \in\bb M_n(\bb C)$ are ensembles of random matrices.
Assume that:
\begin{longlist}
\item
The expression
%
%
\begin{equation}\label{eq9}
\frac1{n^2} \| A_n \|_2^2 + \frac1{n^2}\| B
\|_2^2
\end{equation}
is bounded in probability.

\item
For almost all complex numbers $z$,
\[
\frac1n \log\biggl|\det\biggl(\frac1{\sqrt n} A_n - zI\biggr)\biggr| -
\frac1n \log\biggl|\det\biggl(\frac1{\sqrt n} B_n - zI\biggr)\biggr|
\]
converges in probability to zero. In particular, for each fixed $z$, these
determinants are nonzero with probability $1-o(1)$.
\end{longlist}
Then, $\mu_{(1/{\sqrt n}) A_n} - \mu_{(1/{\sqrt n}) B_n} $ converges
in probability to zero.
\end{theorem}

Note that a stronger version of the above theorem appears in
\cite{TVuK2008}, Theorem 2.1, which additionally gives conditions under
which $\mu_{(1/{\sqrt n}) A_n} - \mu_{(1/{\sqrt n}) B_n} $ converges
almost surely to zero.

The lemma below is a sparse version of~\cite{TVuK2008}, Lemma 1.7.
\begin{lemma}\label{AnBnbdd}
Let $M_n$, $A_n$ and $B_n$ be as in Theorem~\ref{MainThm}. Then
$\frac1{n^2} \| A_n \|^2_2$ and $\int_{\bb C} |z| ^2 \,d\mu_{(1/{\sqrt
n}) A_n} (z)$ are bounded\vspace*{1pt} in probability, and the same
statement holds with $B_n$ replacing $A_n$.
\end{lemma}
\begin{pf}
Our proof is the same as the proof~\cite{TVuK2008}, Lemma 1.7, except
that we
need to use a sparse version of the law of large numbers
(which follows from, e.g.,~\cite{Durrett2010}, Theorem
2.2.6).
%
By the Weyl comparison inequality for second moment\vadjust{\goodbreak}
(see~\cite{TVuK2008}, Lemma A.2) it suffices\vspace*{1pt} to prove that
$\frac1{n^2} \| A_n \|^2_2$ is bounded in probability, and by the
triangle inequality along with inequality~(\ref{Mnbdd}), it thus
suffices to show that $\frac1{n^2} \| X_n \|^2_2$ is bounded in
probability. By the sparse law of large numbers
and the fact that $\bb E|x| ^2 < \infty$, we see that $\frac1{n^2} \|
X_n \|^2_2$ is bounded\vspace*{1pt} in probability. The statement with
$B_n$ replacing $A_n$ is exactly~\cite{TVuK2008}, Lemma 1.7.
\end{pf}

The proof of Theorem~\ref{MainThm} is completed by combining
Theorem~\ref{TVuK21} and Lem\-ma~\ref{AnBnbdd} with the following proposition:
\begin{proposition}\label{Prop22}
Let $0 < \alpha\le1$ be a constant, and let $x$ be a random variable
with mean zero and variance one.
Let $X_n$ be the sparse matrix
ensemble for $x$ with parameter $\alpha$, and let $Y_n$ be the $n$ by
$n$ matrix having i.i.d. copies of
$x$ for each entry (in particular, $Y_n$ is \textup{not} sparse). For
each~$n$,
let $M_n$ be a deterministic $n$ by $n$ matrix satisfying
inequality (\ref{Mnbdd}), and let $A_n:= M_n + X_n$, and let $B_n :=
M_n + Y_n$.
Then, for every fixed $z \in\bb C$, we have that
%
%
\begin{equation}\label{ConvDet}
\frac1n \log\biggl|\det\biggl(\frac1{\sqrt n} A_n - z I\biggr)\biggr| -
\frac1n \log\biggl|\det\biggl(\frac1{\sqrt n} B_n - z I\biggr)\biggr|
\end{equation}
converges in probability to zero.
\end{proposition}

One useful property of the determinant is that it may be computed in a~number
of different ways. In particular, for a matrix $M$, we have
%
%
\begin{equation}\label{ManyDets}
|{\det}(M)| = \prod_{i=1}^n |\lambda_i(M)| =
\prod_{i=1}^n
\sigma_i(M) =
\prod_{i=1}^n \dist(R_i,\Span\{R_1,\ldots,R_{i-1}\}),\hspace*{-35pt}
\end{equation}
where $\lambda_i(M)$ and $\sigma_i(M)$ are the eigenvalues and
singular values
of $M$, respectively, and where $R_i$ denotes the $i$th row of $M$.

In the remainder of the current section, we will outline the program for
proving Proposition~\ref{Prop22} and describe the differences between our
proof and the proof of~\cite{TVuK2008}, Proposition 2.2.
As in~\cite{TVuK2008}, we will prove Proposition~\ref{Prop22} by
writing the
determinant as a product of distances between the $i$th row of a matrix and
the span of the first $i-1$ rows [thanks to (\ref{ManyDets})].
Proposition~\ref{Prop22} can then be proven via three main steps:
\begin{longlist}[(3)]
\item[(1)] A bound on the least singular value due to Tao and Vu
\cite{TVu2008} for
sparse and nonsparse random matrices is used to take care of terms
very high-dimensional subspaces (i.e., span of more than $n-n^{1-\alpha
/6}$ rows).

\item[(2)] Talagrand's inequality is used, along with other ideas from
\cite{TVuK2008}, to take care of terms with high dimension [i.e.,
span of more than $(1-\delta)n$ rows] not already dealt with by the previous
step. Some care must be taken in the sparse case with the constant
$\alpha$
in the exponent in order to use Talagrand's inequality, which is where the
$\alpha/6$ comes from in the previous step.

\item[(3)] A complex version of a result of Chatterjee
\cite{Chatterjee2008}
(namely Theorem~\ref{ComplexChat2008Thm}) along with new ideas in
\cite{TVuK2008} are used to take care of the remaining terms. Here, the
sparse case differs substantially from the nonsparse case, in that we must
use Theorem~\ref{ComplexChat2008Thm} in place of a result due to
Dozier and Silverstein~\cite{DSilverstein2007} used in \cite
{TVuK2008}. This
step, in general, follows Krishnapur~\cite{TVuK2008}, Appendix C, who
investigates a universality principle for nonsparse random matrices
with not
necessarily i.i.d. entries, since there Dozier and Silverstein's result
\cite{DSilverstein2007} cannot be applied.
\end{longlist}
\begin{remark}\label{RDSSparse}
It would be natural to investigate a version of Theorem~\ref{MainThm} where
convergence in the almost sure sense is proved rather than convergence in
probability. Typically, proving almost sure convergence is harder than
proving convergence in probability; however, the universality principle in
\cite{TVuK2008} is proven for both types of convergence, and so may
provide a
general approach to proving a universality principle for sparse random
matrices with almost sure convergence. One of the steps in proving the
universality principle of~\cite{TVuK2008} in the almost sure sense
uses a
result due to Dozier and Silverstein~\cite{DSilverstein2007}. In
\cite{DSilverstein2007}, a truncation argument is used that seems like it
would need to be altered or replaced in order to prove a result for sparse
random matrices.
%
%
Another possible approach to proving a version of Theorem~\ref
{MainThm} for
almost sure convergence would be to prove an analog of
Chatterjee's~\cite{Chatterjee2008}, Theorem 1.1 (see Theorem \ref
{ComplexChat2008Thm}) for
almost sure convergence, though this might require a very different
type of
argument than the one used in~\cite{Chatterjee2008}. A~sparse version
of the
law of large numbers
for almost sure convergence would also likely be necessary in any case.
\end{remark}

\section{\texorpdfstring{Proof of Proposition \protect\ref{Prop22}}{Proof of Proposition 2.3}}
\label{SPropPf}

By shifting $M_n$ by $zI\sqrt n$ [and noting that the new $M_n$ still
satisfies inequality (\ref{Mnbdd})], it is sufficient to prove that
\[
\frac1n \log\biggl|\det\biggl(\frac1{\sqrt n} A_n\biggr)\biggr| -
\frac1n \log\biggl|\det\biggl(\frac1{\sqrt n} B_n\biggr)\biggr|
\]
converges to zero in probability.

Following the notation of~\cite{TVuK2008}, let $X_1,\ldots,X_n$ be
the rows of
$A_n$, and let $Y_1,\ldots,Y_n$ be the rows of $B_n$.
Let $Z_1,\ldots,Z_n$ denote the rows of $M_n$, and note that by
inequality (\ref{Mnbdd}) we have that
\[
\sum_{j=1}^n \| Z_j \|_2^2 = O(n^2).
\]
By re-ordering the rows of $A_n$, $B_n$ and $M_n$ if necessary, we may assume
that the rows $Z_{\lceil n/2 \rceil}, \ldots,Z_n$ have
the smallest
norms, and so
%
%
\begin{equation}\label{eqn28}
\| Z_i \|_2 = O\bigl(\sqrt n\bigr)\qquad \mbox{for } n/2 \le i \le n.
\end{equation}
This fact will be used in part of the proof of
Lemma~\ref{lemHdim}.\vadjust{\goodbreak}

For $1\le i\le n$, let $V_i$ be the $(i-1)$-dimensional space generated by
$X_1,\ldots,\break X_{i-1}$, and let $W_i$ be the $(i-1)$-dimensional space generated
by $Y_1,\ldots,Y_{i-1}$. By standard formulas for the determinant [see
(\ref{ManyDets})], we have that
\[
\frac1n \log\biggl|\det\biggl(\frac1{\sqrt n} A_n \biggr)\biggr| =
\frac1n \sum_{i=1}^n\log\dist\biggl(\frac1{\sqrt n}X_i,V_i\biggr)
\]
and
\[
\frac1n \log\biggl|\det\biggl(\frac1{\sqrt n} B_n \biggr)\biggr| =
\frac1n \sum_{i=1}^n\log\dist\biggl(\frac1{\sqrt n}Y_i,W_i\biggr).
\]
It is thus sufficient to show that
%
%
\begin{equation}\label{goal}
\frac1n \sum_{i=1}^n\log\dist\biggl(\frac1{\sqrt n}X_i,V_i\biggr) -
\log\dist\biggl(\frac1{\sqrt n}Y_i,W_i\biggr)
\end{equation}
converges\vspace*{1pt} in probability to zero.
We will start by proving somewhat weak upper and lower bounds on $\dist
(\frac
1{\sqrt n}X_i,V_i)$ and $\dist(\frac1{\sqrt n}Y_i,W_i)$ that hold for all
$i$. For the upper bound, note that by Chebyshev's inequality we have
$\Pr(\| X_i \|_2 > n^2) \le n^{-3}$, and thus by the
Borel--Cantelli lemma,
we have with probability 1 that $\| X_i \|_2 < n^2$ for
all but
finitely many
$n$ and for all $i$. This implies that, with probability 1,
\[
\dist\biggl(\frac1{\sqrt n} X_i, V_i\biggr) \le\| X_i \|_2 = n^{O(1)}
\]
for all but
finitely many $n$ and for all $i$; and the same bound also holds for
$\dist(\frac1{\sqrt n}Y_i,W_i)$. To show\vspace*{-1pt} a lower bound,
define $S_j^{(i)}:= \Span(\{X_1,\ldots,X_{i}\}\setminus\{X_j\})$,
and define
$A_n^{(i)}$ to be the $i$ by $n$ matrix consisting of the first $i$
rows of
$A_n$. By~\cite{TVuK2008}, Lemma A.4, we have
\[
\sum_{j=1}^{i} \dist\bigl(X_j, S_j^{(i)}\bigr)^{-2} = \sum_{j=1}^i
\sigma_j\bigl(A_n^{(i)}\bigr)^{-2},
\]
and since $V_i=S_i^{(i)}$, we thus have the crude bound
\[
\dist(X_i,V_i)^{-2} \le n \sigma_i\bigl(A_n^{(i)}\bigr)^{-2}.
\]
By Cauchy interlacing (see~\cite{TVuK2008}, Lemma A.1), we know that
$\sigma_i(A_n^{(i)}) \ge\sigma_n(A_n)$, and thus we have
\[
\frac1n \sigma_n(A_n)
\le
\dist\biggl(\frac1{\sqrt n}X_i,V_i\biggr),
\]
and by the same reasoning,
\[
\frac1n \sigma_n(B_n)
\le
\dist\biggl(\frac1{\sqrt n}Y_i,W_i\biggr).\vadjust{\goodbreak}
\]
Lower bounds on $\dist(\frac1{\sqrt n}X_i,V_i)$ and $\dist(\frac
1{\sqrt
n}Y_i,W_i)$ will now follow from~lower bounds on the least singular
values of
$A_n$ and $B_n$ which were proven in~\cite{TVu2008}.

\begin{lemma}[(Least singular value bound for sparse random matrices
\cite{TVu2008})]
\label{SpSingVal}
Let $0<\alpha\le1$ be a constant, and let $x$ be a random variable
with mean zero and variance one.
Let $X_n$ be the sparse matrix
ensemble for $x$ with parameter~$\alpha$, and let $Y_n$ be the $n$ by
$n$ matrix having i.i.d. copies of
$x$ for each entry (in particular, $Y_n$ is \textup{not} sparse). For
each $n$,
let $M_n$ be a deterministic~$n$ by $n$ matrix satisfying
inequality (\ref{Mnbdd}), and let $A_n:= M_n + X_n$, and let $B_n :=
M_n + Y_n$.
Then with probability~1 we have
\[
\sigma_n(A_n),\sigma_n(B_n) \ge n^{-O(1)}
\]
for all but finitely many $n$.
\end{lemma}
\begin{pf}
Paraphrasing~\cite{TVuK2008}, proof of Lemma 4.1, the proof follows by
combining~\cite{TVu2008}, Theorem 2.5 (for the nonsparse matrix) and
\cite{TVu2008}, Theorem~2.9 (for the sparse matrix) each with the
Borel--Cantelli lemma, noting that the hypotheses
of~\cite{TVu2008}, Theorem 2.5, and~\cite{TVu2008}, Theorem 2.9, are
satisfied due
to~\cite{TVu2008}, Lemma 2.4,
and inequality (\ref{Mnbdd}).
\end{pf}

Thus, with probability 1 we have
%
%
\begin{equation}\label{genbd}
\biggl|\log\dist\biggl(\frac1{\sqrt n} X_i,V_i\biggr)\biggr|,
\biggl|\log\dist
\biggl(\frac1{\sqrt n} Y_i,W_i\biggr)\biggr| \le O(\log n)
\end{equation}
for all but finitely many $n$.
In light of inequality (\ref{genbd}), the following two lemmas suffice to
prove that the quantity in display (\ref{goal}) converges in
probability to
zero.

Recall that $\alpha$ is the parameter used to determine the sparseness
of the
sparse matrix ensemble.
\begin{lemma}[(High-dimensional contribution)]\label{lemHdim}
For every $\varepsilon>0$, there exists a constant $0 < \delta
_\varepsilon
< 1/2$
such that for every $0 < \delta< \delta_\varepsilon$ we have with probability
$1$ that
\[
\frac1n \sum_{(1-\delta)n \le i \le n - n^{1-\alpha/6}}
\biggl|\log\dist\biggl( \frac1{\sqrt n} X_i, V_i\biggr)\biggr| = O(\varepsilon)
\]
for all but finitely many $n$.
\end{lemma}

Note that Lemma~\ref{lemHdim} with $Y_i$ (which is not sparse) replacing
$X_i$ and with~$W_i$ replacing $V_i$ was proven in~\cite{TVuK2008},
Lemma 4.2,
with $0.99$ replacing $1-\alpha/6$. Alternatively, the nonsparse case
follows from our proof of Lemma~\ref{lemHdim} if one sets $\alpha=1$ (giving
an exponent of $5/6$ in place of the exponent $0.99$ used in~\cite{TVuK2008}, Lemma~4.2).
Also, note that for all sufficiently large $n$, we may assume that
(\ref{eqn28}) holds for all $i$ relevant to Lemma \ref
{lemHdim} above.
\begin{lemma}[(Low-dimensional contribution)]\label{lemLdim}
For every $\varepsilon> 0$, there exists $0 < \delta< \varepsilon$ such
that with
probability at least $1-O(\varepsilon)$ we have
\[
\biggl|\frac1n \sum_{1\le i \le(1-\delta) n}
\log\biggl(\dist\biggl(\frac 1{\sqrt n} X_i,V_i\biggr)\biggr) -
\log\biggl(\dist\biggl(\frac1{\sqrt n} Y_i,W_i\biggr)\biggr) \biggr| =
O(\varepsilon)
\]
for all but finitely many $n$.
\end{lemma}

To complete the proof of Proposition~\ref{Prop22}, one may combine
Lemma~\ref{lemHdim} (\cite{TVuK2008}, Lem\-ma~4.2) (which is the nonsparse
analog of Lem\-ma~\ref{lemHdim}) and Lem\-ma~\ref{lemLdim}.
In particular, given $\varepsilon_{\Lemmaref{lemHdim}} >0$ in Lemma \ref
{lemHdim}, we
need the $\delta_{\Lemmaref{lemLdim}}$ in Lemma~\ref{lemLdim} to be
smaller than
$\delta_{\varepsilon_{\Lemmarref{lemHdim}}}$. This may\vspace*{1pt} be accomplished by
choosing~$\varepsilon_{\Lemmaref{lemLdim}}$ from Lemma~\ref{lemLdim} to be smaller than
$\delta_{\varepsilon_{\Lemmarref{lemHdim}}}$ given by Lemma~\ref{lemHdim}.

\section{\texorpdfstring{Proof of Lemma \protect\ref{lemHdim}}{Proof of Lemma 3.2}} \label{SLemHdimPf}

Following~\cite{TVuK2008}, we will prove Lemma~\ref{lemHdim} in two parts,
splitting the summands into cases where the log is positive and where
the log
is negative. The proof below follows the proof of~\cite{TVuK2008},
Lemma 4.2,
closely, and we have included it in detail to make explicit the role of
$\alpha$, which determines the sparseness of the matrix $A_n$. One place
where particular care must be taken with sparseness parameter $\alpha$
is in a
truncation argument needed to apply Talagrand's inequality (see
Section~\ref{ssSpLem51}). There, we have made frequent use of the
assumption that $\alpha$ is a positive constant, though it is possible
that a
very slowly decreasing $\alpha$ could also work; see Lemma~\ref{SparseDom}
and Remark~\ref{Rdecalpha}.

\subsection{Positive log component}

In this section, we will use the notation
\[
\poslog(x):= \max\{\log
(x), 0\}.
\]

By the Borel--Cantelli lemma, the desired bound on the positive log component
may be proven by showing
\[
\sum_{n=1}^\infty\Pr\biggl(\frac1n \sum_{(1-\delta)n \le i \le
n-n^{1-\alpha/6}}
\poslog{\dist\biggl( \frac1{\sqrt n} X_i,V_i\biggr)}
\ge\varepsilon\biggr) <
\infty.
\]

We will use the crude bound $
\poslog{\dist( \frac1{\sqrt n} X_i,V_i)}
\le
\poslog{( \frac{\| X_i \|_2}{\sqrt n})}$.
Note that if $ 2^{m_0} \le\frac{\| X_i \|_2}{\sqrt n} < 2^{m_0
+1}$, then $m_0
\le\log_2(\frac{\| X_i \|_2}{\sqrt n}) < m_0+1$, and so
\[
\sum_{m=0}^\infty\mathbh{1}_{\{\| X_i \|_2 \ge
2^m \sqrt n
\}} = m_0 +1 >
\log_2\biggl(\frac{\| X_i \|_2}{\sqrt n}\biggr).
\]
Thus,
%
%
\begin{eqnarray}\label{eqn31}
&&\frac1n \sum_{(1-\delta)n \le i \le n - n^{1-\alpha/6}}
\poslog\dist\biggl( \frac1{\sqrt n} X_i, V_i\biggr)
\nonumber\\[-8pt]\\[-8pt]
&&\qquad\le
\sum_{m=0}^\infty\frac1n \sum_{(1-\delta)n \le i \le n -
n^{1-\alpha/6}} \mathbh{1}_{\{\| X_i \|_2 \ge2^m
\sqrt
n\}}.\nonumber
\end{eqnarray}

If the left-hand side of inequality (\ref{eqn31}) is at least
$\varepsilon$
for a given $n$,
then we must have for some $m \ge0$ that
%
%
\begin{equation}\label{eqn315}
\frac1n \sum_{(1-\delta)n \le i \le n - n^{1-\alpha/6}}
\mathbh{1}_{\{\| X_i \|_2 \ge2^m \sqrt n\}}
\ge\frac
{2\varepsilon}{(100+m)^2}.
\end{equation}

We now have two cases to consider. For the first case, assume that the
smallest $m$ satisfying inequality (\ref{eqn315}) satisfies $m \ge n^{1/5}$.
Then for inequality~(\ref{eqn315}) to be satisfied, there exists some
$1\le
i\le n$ such that $\| X_i \|_2 \ge2^{n^{1/5}} \sqrt{n}$.
By\vspace*{1pt} Chebyshev's
inequality and equation (\ref{eqn28}), we have\vspace*{-1pt} that $\Pr(\| X_i
\|_2
\ge2^{n^{1/5}} \sqrt{n} ) \le
O(\frac1{2^{2n^{1/5}}})$, and thus the probability of such an $i$
existing is
at most $f(n):= 1- (1-c2^{-2n^{1/5}})^n$, where $c$ is some constant.
It is not
hard to show that $f(n)n^2 \to0$ as $n \to\infty$, and thus, for all
sufficiently large $n$, we have the probability that there exists an
$i$ such
that $\| X_i \|_2 \ge2^{n^{1/5}} \sqrt n$ is at most
$\varepsilon
/n^2$. Since
this probability is summable in $n$, we have proved
inequality (\ref{eqn31})
in the first case.

For the second case, assume that the smallest $m$ satisfying
inequality (\ref{eqn315}) satisfies $0 \le m < n^{1/5}$. In this case
we will
use Hoeffding's inequality.
\begin{theorem}[(Hoeffding's inequality~\cite{Hoeffding1963})]
Let $\beta_1,\ldots,\beta_k$ be independent random variables such
that for
$1\le i \le k$ we have
%
\[
\Pr\bigl(\beta_i - \bb E(\beta_i) \in[0,1]\bigr) = 1.
\]
Let $S:= \sum_{i=1}^k \beta_i$. Then
%
\[
\Pr\bigl(S \ge kt + \bb E(S) \bigr) \le\exp( - {2 k t^2}).
\]
\end{theorem}

The random variables $\beta_i$ will be $\mathbh{1}_{\{\|
X_i \|_2 \ge2^m \sqrt n\}}$,
and thus we need to control $\Pr({\| X_i \|_2 \ge2^m
\sqrt n})$ in
order to
bound $\bb E(S)$. By (\ref{eqn28}) and Chebyshev's
inequality, we
have that
%
%
\begin{equation}\label{ESbd}
\Pr\bigl(\| X_i \|_2 \ge2^m \sqrt n\bigr) \le O\biggl(\frac1{2^{2m}}\biggr).
\end{equation}
We will take $k = n-n^{1-\alpha/6} - (1-\delta)n$, so we have that
$\lim_{n\to\infty} \frac kn = \delta$. Also, $\delta_\varepsilon$
sufficiently
small so that $\delta_\varepsilon< \frac{\varepsilon}{20\mbox{,}000C}$, where $C$
is the
implicit constant in inequality (\ref{ESbd}). If we take $t = \frac nk
(\frac{\varepsilon}{(100+m)^2})$,
we can compute that
\[
\frac{kt}n + \frac1n \bb
E(S) \le\frac\varepsilon{(100+m)^2} + \frac{2\delta_\varepsilon C}
{2^{2m}} \le
\frac{2\varepsilon}{(100+m)^2}
\]
for all sufficiently large $n$ (the second
inequality follows by taking $n$ sufficiently large so that $k/n \le2
\delta
< 2\delta_\varepsilon$). Thus, by Hoeffding's inequality and taking $n$
sufficiently large, we have
\begin{eqnarray*}
&&\Pr\biggl(\frac1n \sum_{(1-\delta)n \le i \le n-
n^{1-\alpha/6}} \mathbh{1}_{\{\| X_i \|_2 \ge
2^m\sqrt
n\}} \ge
\frac{2\varepsilon}{(100+m)^2} \biggr) \\
&&\qquad\le\exp\biggl(\frac{- n\varepsilon^2}{\delta(100+m)^4}\biggr) \\
&&\qquad\le\max\biggl\{
\exp\biggl(\frac{-n\varepsilon^2}{\delta(200)^4}\biggr),
\exp\biggl(\frac{-n^{1/5}\varepsilon^2}{16\delta}\biggr)\biggr\},
\end{eqnarray*}
where the last inequality follows from our assumption in this second
case that $0 \le m \le n^{1/5}$. Thus, we have shown for all sufficiently
large $n$ and any $0 \le\delta< \delta_\varepsilon$ that
%
\begin{eqnarray*}
&&\Pr\biggl(\frac1n
\sum_{(1-\delta)n \le i \le n- n^{1-\alpha/6}}
\poslog{\dist\biggl(\frac1{\sqrt n} X_i, V_i\biggr)}
\ge0\biggr) \\
&&\qquad\le\max\biggl\{ \exp
\biggl(\frac{-n\varepsilon^2}{\delta(200)^4}\biggr), \exp
\biggl(\frac{-n^{1/5}\varepsilon^2}{16\delta}\biggr)\biggr\}.
\end{eqnarray*}
%

Finally, we note that the bounds from the two cases sum to at most
\[
\varepsilon/n^2 + \max\biggl\{ \exp\biggl(\frac{-n\varepsilon^2}{\delta
(200)^4}\biggr),
\exp\biggl(\frac{-n^{1/5}\varepsilon^2}{16\delta}\biggr)\biggr\},
\]
which is summable in $n$, thus
completing the proof for the positive log component.

\subsection{Negative log component}

In this section, we will use the notation $\neglog(x):= \max\{-{\log}
(x), 0\}$.

By the Borel--Cantelli lemma, it suffices to show that
%
%
\begin{equation}\label{NegLogGoal}
\sum_{n=1}^\infty\Pr\biggl( \frac1n \sum_{(1-\delta)n\le i \le
n-n^{1-\alpha/6}}
\neglog{\dist\biggl(\frac1{\sqrt n} X_i, V_i\biggr)}
\ge
\varepsilon\biggr) < \infty.
\end{equation}

Following the approach in~\cite{TVuK2008}, our main tool is the following
lemma.
\begin{proposition}\label{SpLem51}
Let $0<\alpha\le1$ be a constant, let $1 \le d \le n- n^{1-\alpha
/6}$, let
$0<c<1$ be a constant and let $W$ be a deterministic $d$-dimensional subspace
of~$\bb C^n$. Let $X$ be a row of $A_n$. Then
\[
\Pr\bigl(\dist(X,W) \le c
\sqrt{n-d}\bigr) \le6\exp(-n^{\alpha/2})
\]
for all $n$ sufficiently large with
respect to $c$ and $\alpha$.
\end{proposition}

We will give the proof of Proposition~\ref{SpLem51} in
Section~\ref{ssSpLem51}. The proof of the negative log component of
Lemma~\ref{lemHdim} can be completed by using\vadjust{\goodbreak} Proposition~\ref{SpLem51} and
following the proof of Lemma 4.2 of~\cite{TVuK2008}, which we paraphrase below.

Taking $c=1/2$ in Proposition~\ref{SpLem51} and conditioning on
$V_i$, we have that for each $(1-\delta)n \le i \le n - n^{\alpha/6}$ that
\[
\Pr\biggl(\dist\biggl(\frac1{\sqrt n} X_i, V_i\biggr) > \frac{\sqrt{n -i+1}}{2\sqrt
n}\biggr) \ge1- O(\exp(n^{-\alpha/2})).
\]
Thus, the probability that
%
%
\begin{equation}\label{DistClose}
\dist\biggl(\frac1{\sqrt n} X_i, V_i\biggr) > \frac{\sqrt{n -i+1}}{2\sqrt n}
\end{equation}
simultaneously for all $(1-\delta)n \le i \le n - n^{\alpha/6}$ is at least
$1- O(n^{-10})$ (in fact, better bounds are possible, but this
is sufficient).

Finally,\vspace*{1pt} choosing $\delta_\varepsilon$ sufficiently small so that
$\frac{\delta_\varepsilon}{2}\log\frac4{\delta_\varepsilon} <
\varepsilon
$, we can
take the log of inequality (\ref{DistClose}) and sum in $i$ to get
that the
probability in the summand of inequality (\ref{NegLogGoal}) in at most
$O(n^{-10})$, and this is summable in $n$, completing the proof of
inequality (\ref{NegLogGoal}).\vspace*{-3pt}

\subsection{\texorpdfstring{Proof of Proposition \protect\ref{SpLem51}}{Proof of Proposition 4.2}}\label{ssSpLem51}

Recall that $X$ has coordinates $a_i = \frac{\Irho
x_i}{\sqrt\rho} + m_i$, where $m_i$ is a fixed element (it comes from the
matrix $M_n$), $x_i$ is a~fixed, mean zero, variance $1$ random
variable (it does not change with $n$) and $\rho= n^{-1+\alpha}$ where $0<\alpha
\le1$ is a constant. The proof of Proposition~\ref{SpLem51} closely follows the proof
of Proposition 5.1 of~\cite{TVuK2008}, and we give the details below to
highlight how the proof must be modified to accommodate sparseness with
parameter $\alpha$. In particular, care must be taken with the value
of~$\alpha$ in the following three steps: first, when reducing to the
case where the sparse random variables are bounded (since sparseness requires
scaling by $1/n^{-1+\alpha}$), second, when showing that the sparse random variables
restricted to the bounded case still have variance tending to 1 as
$n\to\infty$, and third, when applying Talagrand's inequality where
one must keep track of $\alpha$ in the exponent on the upper bound.
\begin{pf*}{Proof of Proposition~\ref{SpLem51}}
First we reduce to the case where $X$ has mean~0. Let $v =
\bb E(X)$. (Note that $v$ is the row of $M_n$ corresponding to $X$.)

Note that $\dist(X,W) \ge\dist(X-v, \Span(W,v))$. Thus, by
changing constants slightly (while still preserving $0<c<1$) and
replacing $d$
by $d+1$, it suffices to prove Proposition~\ref{SpLem51} in the mean zero
case.

The second step is reducing to a case where the coordinates of $X$ are
bounded. In particular, we will\vspace*{1pt} show that, with probability at least $1-
2\exp(-n^{\alpha/2})$, all but $n^{0.8}$ of the coordinates of $X$
take values
that are less than $n^{1/2-\alpha/4}$. Let $t_i:= \mathbh{1}_{\{
|a_i| \ge n^{(1-\alpha/2)/2}\}}$, and let $T:=
\sum
_{i=1}^n t_i$. If $\bb E(T) =
0$, then
with probability $1$ we have that $|a_i| < n^{(1-\alpha/2)/2}$,
and we are
done with the reduction to the case where the coordinates are bounded. Thus,
it is left to show this reduction in the case where $\bb E(T) > 0$.\vadjust{\goodbreak}

By Chernoff (see~\cite{TVuBook}, Corollary 1.9) we know that for every
$\varepsilon>0$ we have
\[
\Pr\bigl(|T - \bb E(T) | \ge\varepsilon\bb E(T)\bigr) \le2 \exp\biggl( -
\min\biggl\{\frac{\varepsilon^2}{4}, \frac\varepsilon{2}\biggr\} \bb E(T)\biggr).
\]
Since $\bb E(T) >0$ by assumption, we may set $\varepsilon:= \frac
{n^{0.8}}{\bb
E(T)} - 1$.
By Chebyshev's inequality, we have\vspace*{1pt} $\Pr(|a_i| \ge
n^{(1-\alpha
/2)/2}) \le
n^{-1 + \alpha/2}$ for all $1\le i \le n$, and thus $\bb E(T) = n \bb E(t_i)
\le n^{\alpha/2}$, which implies that $\varepsilon\ge n^{0.8 - \alpha
/2} -1 \ge
2$ for large~$n$. Here we used the fact that $0 < \alpha/2 \le0.5$. Using
the Chernoff bound we have
\begin{eqnarray*}
\Pr\bigl(T \ge(1+\varepsilon)\bb E(T) = n^{0.8}\bigr) &\le&
2\exp\biggl(-\frac{\varepsilon}{2}\bb E(T)\biggr) \\ 
&\le&2 \exp\bigl(-n^{0.8}/2 + \bb E(T) /2\bigr) \\
&\le&2 \exp\bigl(-n^{0.8}/2 + n^{\alpha/2} /2\bigr) \\
&\le&2 \exp(-n^{0.8}/4) \\ 
& \le&2\exp(-n^{\alpha/2}).
\end{eqnarray*}

Thus,\vspace*{1pt} with probability at least $1 - 2\exp(-n^\alpha)$, there are at most
$n^{0.8}$ indices for which $|a_i| \ge n^{(1-\alpha
/2)/2}$. For a
subset $I \subset\{1,2,\ldots,n\}$, let $E_I$ denote the event that
$I=\{i\dvtx
|a_i| \ge n^{1/2 - \alpha/4}; 1\le i \le n\}$.

By the law of total probability, we have
%
\begin{eqnarray*}
&&\Pr\bigl(\dist(X,W) \le c \sqrt{n -d}\bigr)\\
&&\qquad\le2\exp(-n^{\alpha/2})\\
&&\qquad\quad{}+ \mathop{\sum_{I \subset\{1,\ldots,n\}}}_{|I|
\le n^{0.8}}
\Pr\bigl(\dist(X,W) \le c\sqrt{n-d} | E_I\bigr)\Pr(E_I).
\end{eqnarray*}
%
Thus, it is sufficient to show that
\[
\Pr\bigl(\dist(X,W) \le c\sqrt{n-d} | E_I\bigr)\le4\exp(-n^{\alpha/2})
\]
for each $I\subset\{1,\ldots,n\}$ such that $|I| \le n^{0.8}$.

Fix such a set $I$. By\vspace*{1pt} renaming coordinates, we may assume that $I =
\{n'+1,\ldots,n\}$ where $n-n^{0.8} \le n' \le n$. The next step is
projecting away the coordinates in $I$. In particular, let $\pi\dvtx
\bb C^n
\to\bb C^{n'}$ be the orthogonal projection onto the first $n'$ coordinates,
and note that
\[
\dist(X,W) \ge\dist(\pi(X),\pi(W)).
\]
Thus, we can condition on $a_{n'+1},\ldots, a_n$, adjust $c$ slightly
(without changing the fact that $0<c<1$) and (abusing notation to
henceforth let $n$ stand for~$n'$) see that it is sufficient to show
\begin{eqnarray*}
&&\Pr\bigl(\dist(X,W)\le c\sqrt{n-d} | |a_i| < n^{1/2-\alpha/4}\mbox{,
 for
every } 1\le i \le n\bigr)\\
&&\qquad\le4\exp(-n^{\alpha/2}).
\end{eqnarray*}

\begin{lemma}\label{tilaVar}
Let $\tilai$ be the random variable $a_i$ conditioned on $
|a_i| <
n^{1/2-\alpha/4}$. Then $\tilai$ has variance $1+o(1)$.
\end{lemma}
\begin{pf}
By definition
\begin{eqnarray*}
\Var(\tilai) &=& \bb E(|\tilai|^2) - |\bb E(\tilai
)|^2 \\
&=& \bb E(|a_i|^2 | |a_i| < n^{1/2-\alpha/4})
- \bigl|\bb E(a_i | |a_i| < n^{1/2-\alpha/4})\bigr|^2
\\
&=& \frac{1}{\Pr(|a_i| < n^{1/2-\alpha/4})} \bb E\bigl(
|a_i|^2
\mathbh{1}_{\{\aibdd\}}\bigr) \\
&&{} - \frac1{\Pr(\aibdd)^2} \bigl|\bb E\bigl(a_i
\mathbh{1}_{\{\aibdd\}}\bigr)\bigr|^2.
\end{eqnarray*}

Note that $a_i = \frac{\Irho x_i}{\sqrt\rho}$, and so $\aibdd$ if
and only
if $|\Irho x_i| < n^{\alpha/4}$. Since~$x_i$ does not
change with $n$,
we see\vspace*{1pt} that $\Pr(\aibdd) = \Pr(|\Irho x_i| < n^{\alpha
/4}) \to 1$ as $n\to\infty$.
Also, by Lemma~\ref{SparseDom}, we know that $\bb E(|a_i|^2
\mathbh{1}_{\{\aibdd\}}) \to\bb E(|a_i|^2)
= 1$
and that $\bb E(a_i \mathbh{1}_{\{\aibdd\}})\to
\bb E(a_i) = 0$.
Thus,\vspace*{1pt} we have shown that $\tilai$ has variance $1+o(1)$.
\end{pf}

Next, we recenter $\tilai$ by subtracting away its mean, and we
call the result~$\tilai$. Note that this recentering does not change the
variance. We will use the following version of Talagrand's inequality,
quoted from~\cite{TVuK2008}, Theorem~5.2; see also
\cite{Ledoux2001}, Corollary 4.10: 
\begin{theorem}[(Talagrand's inequality)]\label{TalThm}
Let $\mathbf D$ be the unit disk $\{z\in\bb C,\break |z| \le1\}$.
For every
product probability $\mu$ on ${\mathbf D}^n$, every convex
$1$-Lipschitz function
$F\dvtx\bb C^n \to\bb R$, and every $r \ge0$,
\[
\mu\bigl(|F- M(F)| \ge r \bigr) \le4 \exp(-r^2/8),
\]
where $M(F)$ denotes the median of $F$.
\end{theorem}

Let $\tilX= (\tilde{a}_{1}, \tilde{a}_{2}, \ldots,\tilde{a}_{n})$,
and let $\mu$ be the
distribution on ${\mathbf D}^n$ given by $\tilX/2n^{1/2-\alpha/4}$.
%
Let $F(u):= \frac12 \dist(u, W)^2$, and
note that $F$ is convex and $1$-Lipschitz, which follows since
$\dist(u, W)$
is both convex and $1$-Lipschitz [and
also using the fact that
$\dist(u, W)\le1$,
since $0\in W$].\vspace*{1pt}

By Theorem~\ref{TalThm} with $r = 3 n^{\alpha/4}$, we have
\[
\Pr\bigl(|{\dist}(\tilX, W)^2 - M(\dist(\tilX,W)^2)| \ge12
n^{\alpha/4} n^{1-\alpha/2}\bigr) \le4 \exp(-n^{\alpha/2}),
\]
which implies that
%
%
\begin{equation}\label{distM}
\Pr\bigl(\dist(\tilX,W)^2 \le M(\dist(\tilX,W)^2) - 12 n^{1-\alpha/4}\bigr)
\le4
\exp(-n^{\alpha/2}).
\end{equation}

Recall\vspace*{1pt} that $F= \frac12 \dist(\frac{\tilX}{2n^{1/2-\alpha/4}}, W)^2$.
Using Talagrand's inequality (Theorem~\ref{TalThm}) again, we will
show that
the mean of $F$ is very close to the median of~$F$. We compute
\begin{eqnarray*}
|\bb E(F) - M(F)| & \le & \bb E|F - M(F)| =
\int_0^\infty
\Pr\bigl(|F-M(F)| \ge t\bigr) \,dt\\
&\le&\int_0^\infty4 \exp(-t^2/8) \,dt = 8\sqrt{2\pi}.
\end{eqnarray*}
Thus, we have shown that
%
%
\begin{equation}\label{closeEM}
|\bb E(\dist(\tilX,W)^2) - M (\dist(\tilX,W)^2)| \le
\bigl(32\sqrt
{2\pi}\bigr)
n^{1-\alpha/2}.
\end{equation}

\begin{lemma}\label{SparseLem53}
$\bb E(\dist(\tilX,W)^2) = (1+o(1))(n-d)$.
\end{lemma}
\begin{pf}
Let $\pi:=(\pi_{ij})$ denote the orthogonal projection matrix to $W$. Note
that $\dist(\tilX,W)^2 = \sum_{i=1}^n\sum_{j=1}^n \tilai\pi_{ij}
\tilde{a}_{j}$.
Since $\tilai$ are i.i.d., mean zero random variables, we have
\[
\bb E(\dist(\tilX,W)^2) = \bb E(|\tilai|^2) \sum
_{i=1}^n\pi
_{ii} = \bb
E(|\tilai|^2) \tr(\pi).
\]
The proof is completed by applying Lemma~\ref{tilaVar} and
noting that the trace of $\pi$ is $n-d$.
\end{pf}

From inequality (\ref{distM}), we see that it is sufficient to show that
\[
M(\dist(\tilX,W)^2) - 12n^{1-\alpha/4} \ge c^2(n-d).
\]
Using inequality (\ref{closeEM}) and Lemma~\ref{SparseLem53} we have for
sufficiently large $n$ that
\begin{eqnarray*}
&&M(\dist(\tilX,W)^2) - 12 n^{1-\alpha/4} \\
&&\qquad\ge\bb E(\dist(\tilX,W)^2) -
\bigl(32\sqrt{2\pi}\bigr) n^{1-\alpha/2} - 12 n^{1-\alpha/4} \\
&&\qquad\ge\biggl(c^2 + \frac{1-c^2}2\biggr)(n-d) -\bigl(32\sqrt{2\pi}\bigr) n^{1-\alpha/2} -
12n^{1-\alpha/4} \\
&&\qquad\ge c^2(n-d) + \biggl(\frac{1-c^2}{2}\biggr)n^{1-\alpha/6} -
\bigl(32\sqrt{2\pi}\bigr)
n^{1-\alpha/2}
- 12n^{1-\alpha/4}\\
&&\qquad\ge c^2(n-d),
\end{eqnarray*}
where the last inequality follows from the fact that
\[
\biggl(\frac{1-c^2}{2}\biggr)n^{1-\alpha/6} - \bigl(32\sqrt{2\pi}\bigr)
n^{1-\alpha/2}
- 12n^{1-\alpha/4}\vadjust{\goodbreak}
\]
is a positive quantity for sufficiently large $n$.
Combining the above computation with inequality (\ref{distM}) completes the
proof of Proposition~\ref{SpLem51}.
\end{pf*}

\section{\texorpdfstring{Proof of Lemma \protect\ref{lemLdim}}{Proof of Lemma 3.3}} \label{SLemLdimPf}

Lemma~\ref{lemLdim} follows directly from the slightly more detailed statement
in Lemma~\ref{lemLdimm} given below. In this section, we will prove
Lem\-ma~\ref{lemLdimm} by adapting the proof of Lemma 4.3 of~\cite{TVuK2008},
with some changes. The biggest difference with the proof of Lemma 4.3
of~\cite{TVuK2008} is in the proof of Lemma~\ref{Lmedt},
where we must
adapt the approach of Krishnapur from~\cite{TVuK2008}, Appendix~C, to
a sparse
setting (see Lemma~\ref{KrishLem}). This is one critical juncture
where it
seems like it would take a new idea to prove almost sure convergence in place
of convergence in probability. One possible approach would be proving a
sparse version of~\cite{DSilverstein2007} (which is used in \cite
{TVuK2008} in
the proof of almost sure convergence in the nonsparse case).
Other notable differences from the proof of Lemma
4.3 of~\cite{TVuK2008} are
that we must use Proposition~\ref{SpLem51} in place of Proposition 5.1
of~\cite{TVuK2008}, and that we kill keep track of a~lower bound
on $\delta$, which simplifies some steps in the proof.
\begin{lemma}\label{lemLdimm}
For every $\varepsilon_1 >0$ and for all sufficiently small $\varepsilon
_2>0$, where~$\varepsilon_2$ depends on $\varepsilon_1$ and other constants, the
following holds.
For every $\delta>0$ satisfying
\[
\varepsilon_2^2 < \delta\le\frac{\varepsilon_2}{40 \log(1/\varepsilon_2)},
\]
we have with probability $1 - O(\varepsilon_1)$
\[
\biggl|\frac1n \sum_{1 \le i \le(1-\delta)n} \log\dist\biggl(\frac
1{\sqrt n} X_i, V_i\biggr) - \log\dist\biggl(\frac1{\sqrt n} Y_i, W_i\biggr)
\biggr|
= O(\varepsilon_2)
\]
for all but finitely many $n$.
\end{lemma}


As shown in~\cite{TVuK2008}, Section 6, it is sufficient to prove that with
probability $1-O(\varepsilon_1)$ we have
%
%
\begin{equation}\label{eqn35}
\Biggl|\frac1{n'} \sum_{i=1}^{n'} \log\biggl( \frac1{\sqrt n} \sigma
_i(A_{n,n'})\biggr) - \log\biggl( \frac1{\sqrt n} \sigma_i(B_{n,n'})\biggr)
\Biggr|=O(\varepsilon_2)
\end{equation}
for all but finitely many $n$,
where $n' = \lfloor(1-\delta)n \rfloor$, where $\sigma
_i(A)$ denotes the $i$th
largest singular value of a matrix $A$, and where $A_{n,n'}$ denotes the
matrix consisting of the first $n'$ rows of $A_n$ and $B_{n,n'}$
denotes the
matrix consisting of the first $n'$ rows of $B_n$.

Proving (\ref{eqn35}) is equivalent to showing
%
%
\begin{equation}\label{eqn35alt}
\biggl|\int_0^\infty\log t \,d\nu_{n,n'}(t)\biggr| = O(\varepsilon_2),
\end{equation}
where $d\nu_{n,n'}$ is defined by the difference of the two relevant ESDs,
namely
\[
d\nu_{n,n'} = d\mu_{A_{n,n'}A_{n,n'}^*/{n'}}
-
d\mu_{B_{n,n'}B_{n,n'}^*/{n'}}.\vadjust{\goodbreak}
\]

Following~\cite{TVuK2008}, we can prove (\ref{eqn35alt}) by dividing
the range of $t$ into a few parts, which follows from Lemma~\ref{Llarget}
(for large $t$), Lemma~\ref{Lmedt} (for intermedi\-ate-sized $t$) and
Lemma~\ref{Lsmallt}
(for small $t$).\vspace*{-2pt}
\begin{lemma}[(Region of large $t$)]\label{Llarget}
For every $\varepsilon_1>0$, there exist constants $\varepsilon_2>0$ and
$R_{\varepsilon_2}$ such that with probability $1-O(\varepsilon_1)$ we have
\[
\int_{R_{\varepsilon_2}}^\infty|{\log t}| |d\nu
_{n,n'}(t)|
\le\varepsilon_2.\vspace*{-2pt}
\]
\end{lemma}

\begin{pf}
By Lemma~\ref{AnBnbdd} and Lemma A.2 of~\cite{TVuK2008}, we have that
$\int_0^\infty t |d\nu_{n,n'}(t)|$
is bounded in probability. Thus, there exists a
constant $C_{\varepsilon_1}$ depending on~$\varepsilon_1$
such that with probability $1-O(\varepsilon_1)$ we have
\[
\int_0^\infty t |d\nu_{n,n'}(t)| \le C_{\varepsilon_1}.
\]
Choose $\varepsilon_2>0$ sufficiently small with respect to $\varepsilon
_1$ and
$C_{\varepsilon_1}$ so that
\[
1 \ge2 C_{\varepsilon_1} \varepsilon_2 \log\biggl(\frac{1}{\varepsilon
_2}\biggr).
\]
Set $R_{\varepsilon_2} = (\frac{1}{\varepsilon_2})^2$, and
assume without
loss of
generality that $R_{\varepsilon_2} > e$. Note~that $\frac{t}{\log t}$ is
increasing for $t\ge R_{\varepsilon_2} > e$, and thus by the definition
of $\varepsilon_2$ we
have
\[
\frac{C_{\varepsilon_1}}{\varepsilon_2} \log(t) \le t,
\]
whenever $t \ge R_{\varepsilon_2}$. Thus, we have with probability
$1-O(\varepsilon_1)$ that
\[
\int_{R_{\varepsilon_2}}^\infty|{\log t}| |d\nu
_{n,n'}(t)|
\le\int_{0}^\infty\frac{\varepsilon_2}{C_{\varepsilon_1}} t
|d\nu_{n,n'}(t)|
\le\varepsilon_2.\vspace*{-2pt}
\]
\upqed\end{pf}
\begin{lemma}[(Region of intermediate $t$, namely $\varepsilon_2^2 \le t
\le
R_{\varepsilon_2}$)]\label{Lmedt}
Define a~smooth function $\psi(t)$ which equals $1$ on
the interval $[\varepsilon_2^4, R_{\varepsilon_2}]$,
equals zero outside the interval $(\varepsilon_2^4/2, 2 R_{\varepsilon_2})$,
is monotonically increasing on
$(\varepsilon_2^4/2,\varepsilon_2^4)$ and is monotonically decreasing on
$( R_{\varepsilon_2},2 R_{\varepsilon_2})$.

Then with probability $1-O(\varepsilon_1)$ we have
\[
\biggl|\int_0^{\infty} \psi(t) \log(t) \,d\nu_{n,n'}(t)\biggr| =
O(\varepsilon_2),
\]
so long as $\delta\le\frac{\varepsilon_2}{40 \log(1/\varepsilon_2)}$.\vspace*{-2pt}
\end{lemma}

The main step in this proof is applying Lemma~\ref{KrishLem}, whereas
in the
analogous step in the nonsparse case,~\cite{TVuK2008} uses a result
of Dozier
and Silverstein~\cite{DSilverstein2007}, which proves almost sure convergence
of the relevant distributions (rather than convergence in probability, which
is the limit of Lemma~\ref{KrishLem}). It would be interesting to see
if a
sparse analog of~\cite{DSilverstein2007} is possible,
especially as it might
be a step toward proving a universality result for sparse random matrices
with almost sure convergence instead of convergence in probability.\vadjust{\goodbreak}
\begin{pf*}{Proof of Lemma~\ref{Lmedt}}
Using~\cite{TVuK2008}, Lemma A.1, and the upper bound on $\delta$, it is
possible to show that
\[
\biggl|\int_0^{\infty} \psi(t) \log(t) \,d\nu_{n,n'}(t)\biggr| =
\biggl|\int_0^{\infty} \psi(t) \log(t) \,d\nu_{n,n}(t)\biggr| +
O(\varepsilon_2).
\]
(A possible alternative to the step above would be proving an analog of
Lem\-ma~\ref{KrishLem} for rectangular $n$ by $n'$ matrices.)

By Lemma~\ref{KrishLem} (see Section~\ref{ChattApp}), we know that
$d\nu_{n,n}$ converges in probability to zero, and thus
\[
\biggl|\int_0^{\infty} \psi(t) \log(t) \,d\nu_{n,n}(t)\biggr| =
O(\varepsilon_2),
\]
completing the proof.
\end{pf*}

The last step in proving (\ref{eqn35alt}) and thus completing the
proof of Lemma~\ref{lemLdimm} is the following lemma:
\begin{lemma}[(Region of small $t$, namely $0<t\le\varepsilon_2^4 <
\delta^2$)]\label{Lsmallt}
With probability~$1$, we have
\[
\int_{0}^{\varepsilon_2^4} |{\log t}| |d\nu
_{n,n'}(t)| =
O(\varepsilon_2),
\]
so long as $\delta\le\frac12 (\frac{\varepsilon_2}{\log
(1/\varepsilon_2)})^{1/4}$.
\end{lemma}
\begin{pf}
The required upper bound on $\delta$ follows from the assumption that
$\delta
< \frac{\varepsilon_2}{40\log(1/\varepsilon_2)}$. The proof\vspace*{1pt} is the same
as the
proof for Lemma 6.6 of~\cite{TVuK2008}, with the small change that one
must use
Proposition~\ref{SpLem51} in place of Proposition 5.1 of~\cite{TVuK2008}.
\end{pf}

\subsection{Applying an approach of Chatterjee}\label{ChattApp}

In this subsection, we follow the ideas used by Krishnapur in~\cite{TVuK2008},
Appendix C, where a central-limit-type theorem due to Chatterjee
\cite{Chatterjee2008} was used to prove a universality result for random
matrices with independent but not necessarily identically distributed entries.
Lemma~\ref{KrishLem} below is analog of Lemma C.3 of~\cite{TVuK2008}. Recall
that $\Irho$ is an i.i.d. copy of the random variable taking the value
1 with
probability $\rho$ and the value 0 with probability $1-\rho$, where
$\rho=
n^{-1+\alpha}$ where $0<\alpha\le1$ is a positive constant.
\begin{lemma}\label{KrishLem}
Let $x$ be a complex random variable with mean zero and variance one.
Let ${\mathbf X}=
(X_{1,1}^{(0)},X_{1,1}^{(1)},X_{2,1}^{(0)},X_{2,1}^{(1)},\ldots)$
be an array of $2n^2$ real random variables, where for each $1\le
i,j\le n$ we
define $X_{i,j}^{(0)}$ and $X_{i,j}^{(1)}$ so that\vspace*{1pt} $X_{i,j}^{(0)} + \ci
X_{i,j}^{(1)}$ is an
i.i.d. copy of
$x\Irho/\sqrt\rho$.
Similarly, let $\mathbf{Y}=
(Y_{1,1}^{(0)},Y_{1,1}^{(1)},\break Y_{2,1}^{(0)},Y_{2,1}^{(1)},\ldots)$ be
another\vadjust{\goodbreak}
array of $2n^2$ real random variables, where for each $1\le i,j\le n$ we
define $Y_{i,j}^{(0)}$ and $Y_{i,j}^{(1)}$ so that $Y_{i,j}^{(0)} + \ci
Y_{i,j}^{(1)}$ is an
i.i.d. copy of
$x$ (thus, the $\Xijk$ are sparse versions of the $\Yijk$, which are not
sparse).
Let $A_n({\mathbf X})$ denote the $n$ by $n$ random matrix having
$X_{i,j}^{(0)} +
\ci X_{i,j}^{(1)}$ for the $(i,j)$ entry, and similarly for $A_n(\mathbf{Y})$.
Let $\mu_{(1/n) A_n({\mathbf X})A_n({\mathbf X})^*}$ and $\mu_{
(1/n)
A_n(\mathbf{Y})A_n(\mathbf{Y})^*}$ denote\vspace*{1pt} the ESDs of $\frac1 n
A_n({\mathbf X}
)A_n({\mathbf X})^*$ and
$\frac1 n A_n(\mathbf{Y})A_n(\mathbf{Y})^*$,\vspace*{1pt} respectively.
Then $\mu_{(1/n)A_n(X)A_n(X)^*} - \mu_{(1/n)A_n(Y)A_n(Y)^*} $
converges in probability to zero as \mbox{$n \to\infty$}.
\end{lemma}
\begin{pf}
Our approach will be applying~\cite{Chatterjee2008}, Theorem 1.1, in a
similar way to~\cite{TVuK2008}, Lemma C.3.

Let $H_n({\mathbf X}):= \pmatrix{ {0 \atop A_n({\mathbf X})/{\sqrt n}}
\enskip
{A_n({\mathbf X})^*/{\sqrt n}
\atop0} }$.\vspace*{2pt}

Note that the eigenvalues of $H_n({\mathbf X})$ with multiplicity are exactly
the positive and negative square roots of the eigenvalues with
multiplicity of
$\frac1n A_n({\mathbf X})A_n({\mathbf X})^*$. Also, the same fact
applies to
$H_n(\mathbf{Y})$
and $\frac1 n A_n(\mathbf{Y})A_n(\mathbf{Y})^*$.
We will now follow the computation given in~\cite{Chatterjee2008}, Section
2.4.
It is sufficient to show that $\mu_{H_n({\mathbf X})}-\mu_{H_n(\mathbf
{Y})} $
converges in
probability to zero as $n \to\infty$.

Let $u,v \in\bb R$ with $v \ne0$ and let $z = u + \ci v$. Define a function
\mbox{$f\dvtx\bb R^{2n^2}\to\bb C$} by
\[
f(\mathbf{x}) = \frac1 {2n} \tr\bigl(\bigl(H_n(\mathbf{x}) - z I\bigr)^{-1}\bigr).
\]
Here $\mathbf{x}= (x_{i,j}^{(k)})_{1\le i,j \le n; k \in\{0,1\}}$,
where $x_{i,j}^{(0)}$
corresponds to the real part (namely, $X_{i,j}^{(0)}$ or
$Y_{i,j}^{(0)}$) and
$x_{i,j}^{(1)}$ corresponds\vspace*{2pt} to the complex part (namely,
$X_{i,j}^{(1)}$ or~$Y_{i,j}^{(1)}$). We will show that for every fixed complex $z$ with
$\Imm(z)= {v} \ne0$, we have $\bb E(f(X)) -\bb E(f(F)) \to0$ as
$n\to
\infty$, which implies that $\mu_{H_n({\mathbf X})}-\mu_{H_n(\mathbf
{Y})} $
converges in
probability to zero as $n \to\infty$.

Define $G\dvtx\bb R^{2n^2} \to\bb C^{(2n)^2}$ by
\[
G(\mathbf{x}) = \bigl(H_n(\mathbf{x}) - z I\bigr)^{-1}.
\]

All eigenvalues of $H_n(\mathbf{x})$ are real, and thus all
eigenvalues of
$H(\mathbf{x})-zI$ are nonzero (since $v \ne0$). Thus, $G(\mathbf
{x})$ is well defined.
From the matrix inversion formula, each entry of $G(\mathbf{x})$ is a rational
expression\vspace*{-1pt} in $x_{i,j}^{(k)}$ for $1\le i,j\le n$ and $k \in\{0,1\}$. Thus
$G$ is infinitely differentiable in each coordinate
$x_{i,j}^{(k)}$.

In the remainder of this section, we will use the shorthand $G$ for
$G(\mathbf{x})$
and the shorthand $H$ for $H_n(\mathbf{x})$. Our goal is to apply the
approach used
by Chatterjee in~\cite{Chatterjee2008}, and we will first establish
useful bounds on the partial derivatives of $G$.

Note that
%
%
\begin{equation}\label{eGdiff}
\frac{\partial G}{\partial x_{i,j}^{(k)}} = - G\,\frac{\partial
H}{\partial{x_{i,j}^{(k)}}} G
\end{equation}
[this can be seen by using the product rule and differentiating both
sides of
the equation $(H_n(\mathbf{x}) - zI)G = I$].
The following three formulas follow from~(\ref{eGdiff}) and
the fact
that $\tr(AB) = \tr(BA)$ for any two square matrices~$A$ and~$B$,
along with
the fact that all higher partial derivatives of $H$ are zero.
\fontsize{10pt}{\baselineskip}\selectfont{
\begin{eqnarray*}
\frac{\partial f}{\partial x_{i,j}^{(k)}} & = &\frac{-1}{2n} \tr
\biggl(\frac{\partial
H}{\partial{x_{i,j}^{(k)}}} G^2\biggr), \\
\frac{\partial^2 f}{\partial x_{i_{1},j_{1}}^{(k_{1})}
\,\partial x_{i_{2},j_{2}}^{(k_{2})}} & = &\frac{1}{2n}\biggl( \tr\biggl(
\frac{\partial H}{\partial{x_{i_1,j_1}^{(k_1)}}} G\,
\frac{\partial H}{\partial{x_{i_2,j_2}^{(k_2)}}} G^2\biggr)+\tr
\biggl(\frac{\partial H}{\partial{x_{i_2,j_2}^{(k_2)}}} G\,\frac{\partial
H}{\partial{x_{i_1,j_1}^{(k_1)}}} G^2\biggr)
\biggr),\\
\frac{\partial^3 f}{\partial x_{i_{1},j_{1}}^{(k_{1})}\,\partial
x_{i_{2},j_{2}}^{(k_{2})}\,\partial x_{i_{3},j_{3}}^{(k_{3})}} & = &\frac
{-1}{2n}\sum_{\sigma\in S_3} \tr\biggl(
\frac{\partial H}{\partial{x_{i_{\sigma(1)},j_{\sigma
(1)}}^{(k_{\sigma(1)})}}}
G\,\frac{\partial H}{\partial{x_{i_{\sigma(2)},j_{\sigma
(2)}}^{(k_{\sigma(2)})}}}
G\,\frac{\partial H}{\partial{x_{i_{\sigma(3)},j_{\sigma
(3)}}^{(k_{\sigma(3)})}}}
G^2\biggr),
\end{eqnarray*}}\normalsize
where the last sum is over the six elements of $S_3$, the symmetric
group on 3
letters.

As in~\cite{Chatterjee2008}, Section 2.4, we will use the following facts
to bound the partial derivatives of $f$. Recall that for a matrix $A$, we
define $\|A\|_2:=\tr(AA^*)$. Note that $|{\tr}(AB)|\le\|
A\|_2 \|
B\|_2$.
Also, for $A$ a $k$ by $k$ normal matrix with eigenvalues
$\lambda_1,\lambda_2,\ldots,\lambda_k$ and $B$ any square matrix,
we have
$\max\{ \|AB\|_2, \|BA\|_2\} \le( \max_{1\le i \le k} \lambda_i)
\|B\|_2$.

By the\vspace*{2pt} definition of $G$, it is clear that the absolute value of the largest
eigenvalue of $G$ is at most $|v|^{-1}$. Also,\vspace*{-1pt} by the definition
of $H$,
it is clear that $\frac{\partial H}{\partial{x_{i,j}^{(k)}}}$ is the
matrix having $(\ci)^k n^{-1/2}$ for the
$(n+i,j)$ entry, having $(-\ci)^k n^{-1/2}$ for the $(j,n+i)$ entry and
having zero for all other entries.

Thus, for all $1\le i,j\le n$ and $k \in\{0,1\}$, we have that
\begin{eqnarray*}
\biggl|\tr\biggl(\frac{\partial H}{\partial{x_{i,j}^{(k)}}} G^2\biggr)
\biggr|&\le&\biggl\| \frac{\partial H}{\partial{x_{i,j}^{(k)}}} \biggr\|_2
\| G^2 \|_2 \\
&\le&\sqrt{\frac2 n} |v|^{-2} \| I \|_2 \\
&\le&|v|^{-2}\sqrt2,
\end{eqnarray*}
and so $ |\frac{\partial f}{\partial x_{i,j}^{(k)}}|
<
\frac{|v|^{-2}}{n}$.\vspace*{1pt}

By similar means, we can compute
\begin{eqnarray*}
\biggl|\tr\biggl(\frac{\partial H}{\partial{x_{i_1,j_1}^{(k_1)}}} G\,\frac
{\partial H}{\partial{x_{i_2,j_2}^{(k_2)}}} G^2\biggr)\biggr| &\le&\biggl\|
\frac{\partial H}{\partial{x_{i_1,j_1}^{(k_1)}}} \biggr\|_2
\biggl\| G \,\frac{\partial H}{\partial{x_{i_2,j_2}^{(k_2)}}} G^2
\biggr\|_2 \\
&\le&\sqrt{\frac2 n } |v|^{-3} \sqrt{\frac2 n} \\
&\le&\frac{2|v|^{-3}}{n},
\end{eqnarray*}
which shows that
$ |\frac{\partial^2 f}{\partial x_{i_{1},j_{1}}^{(k_{1})}
\,\partial x_{i_{2},j_{2}}^{(k_{2})}}| \le\frac{2|v|^{-3}}{n^2}$;
and
\begin{eqnarray*}
\biggl|\tr\biggl(\frac{\partial H}{\partial{x_{i_1,j_1}^{(k_1)}}} G\,\frac
{\partial H}{\partial{x_{i_2,j_2}^{(k_2)}}} G \,\frac{\partial
H}{\partial{x_{i_3,j_3}^{(k_3)}}} G^2\biggr)\biggr|
&\le&\biggl\| \frac{\partial H}{\partial{x_{i_1,j_1}^{(k_1)}}}
\biggr\|_2
\biggl\| G\,\frac{\partial H}{\partial{x_{i_2,j_2}^{(k_2)}}} G\,\frac
{\partial H}{\partial{x_{i_3,j_3}^{(k_3)}}} G^2 \biggr\|_2 \\
&\le&\sqrt{\frac2 n } |v|^{-1} \frac{1}{\sqrt n}
|v|^{-3}\sqrt{\frac2 n} \\
&\le&\frac{2|v|^{-4}}{n^{3/2}},
\end{eqnarray*}
which shows that
$|\frac{\partial^3 f}{\partial
x_{i_{1},j_{1}}^{(k_{1})}\,\partial x_{i_{2},j_{2}}^{(k_{2})}\,\partial
x_{i_{3},j_{3}}^{(k_{3})}}| \le\frac
{6|v|^{-4}}{n^{5/2}} $.\vspace*{1pt}


We will now apply a complex version of the main theorem from
\cite{Chatterjee2008}. First, we need the following definitions for a
function $h\dvtx\bb R^N \to\bb C$. We define the \textit{derivative-product
degree with respect to $h$} of a monomial of partial derivatives of $h$
to be
the sum of the number of partial derivatives taken in each factor when the
monomial is written as a product of linear terms. We will use
\textit{derivative-product degree} when the function $h$ is understood. For
example, the derivative-product degree of $(\frac{\partial h}{\partial
x_{i,j}^{(k)}})^3=(\frac{\partial h}{\partial x_{i,j}^{(k)}})(\frac
{\partial h}{\partial x_{i,j}^{(k)}})(\frac{\partial h}{\partial
x_{i,j}^{(k)}})$ is $3$, and the derivative product degree
of\vspace*{1pt} $(\frac{\partial^3 h}{\partial x_{i_{1},j_{1}}^{(k_{1})}\,\partial
x_{i_{2},j_{2}}^{(k_{2})}\,\partial x_{i_{3},j_{3}}^{(k_{3})}})^2(\frac
{\partial^2 h}{\partial x_{i_{3},j_{3}}^{(k_{3})}\,
\partial x_{i_{4},j_{4}}^{(k_{4})}})=(\frac{\partial^3 h}{\partial
x_{i_{1},j_{1}}^{(k_{1})}\,\partial x_{i_{2},j_{2}}^{(k_{2})}\,\partial
x_{i_{3},j_{3}}^{(k_{3})}})(\frac{\partial^3 h}{\partial
x_{i_{1},j_{1}}^{(k_{1})}\,\partial x_{i_{2},j_{2}}^{(k_{2})}\,\partial
x_{i_{3},j_{3}}^{(k_{3})}})\times(\frac{\partial^2 h}{\partial
x_{i_{3},j_{3}}^{(k_{3})}\, \partial x_{i_{4},j_{4}}^{(k_{4})}})$ is $8$.
Define two quantities as follows:
%
\[
\lambda_2(h):=\sup_{x\in\bb R^N}\{\gamma(x): \mbox{$\gamma$
has derivative-product degree 2 with respect to $h$}\}
\]
and let
\[
\lambda_3(h):=\sup_{x\in\bb R^N}\{\gamma(x): \mbox{$\gamma$
has derivative-product degree 3 with respect to $h$}\}.
\]

\begin{theorem}\label{ComplexChat2008Thm}
Let $N$ be a positive even integer, let ${\mathbf X}= (X_1,\ldots,
X_N)$ and
$\mathbf{Y}=
(Y_1,\ldots, Y_N)$ be lists of real-valued random variables such that for
$1\le\ell\le N/2$, the random variables $X_{2\ell-1}$ and $X_{2\ell
}$ are
each independent of all $X_j$ such that $1\le j \le N$ and $j \notin\{
2\ell
-1, 2\ell\}$, and similarly the random variables $Y_{2\ell-1}$ and
$Y_{2\ell}$
are each independent of all $Y_j$ such that $1\le j \le N$ and $j
\notin
\{2\ell-1, 2\ell\}$. Assume further that
%
%
\begin{eqnarray}\label{NewChatConds}
\bb E( X_j) &=& \bb E( Y_j) \qquad\mbox{for all } 1 \le j \le N,
\nonumber\\
\bb E( X_j^2) &=& \bb E( Y_j^2) \qquad\mbox{for all } 1 \le j \le N,
\\
\bb E( X_{2\ell-1}X_{2\ell}) &=& \bb E( Y_{2\ell-1}Y_{2\ell})
\qquad\mbox{for all } 1 \le\ell\le N/2.\nonumber
\end{eqnarray}

Let $h\dvtx\bb R^N \to\bb R$ have continuous partial derivatives of order
1, 2
and 3, including mixed partial derivatives.\vadjust{\goodbreak} If we set $U = h({\mathbf X})$
and $V =
h(\mathbf{Y})$, then for any thrice differentiable $g\dvtx\bb R \to\bb
R$ and
any $K >
0$,
\begin{eqnarray*}
&&|\bb E g(U) - \bb E g(V)|\\[-2pt]
&&\qquad\le
C_1(g) \lambda_2(h) \sumiNodd
\bigl(
\bb E (X_i^2 + |X_iX_{i+1}|+X_{i+1}^2; |X_i|+|X_{i+1}| >
K) \\[-2pt]
&&\qquad\quad\hspace*{74.4pt}{} + \bb E (Y_i^2 + |Y_iY_{i+1}|+Y_{i+1}^2;
|Y_i|+|Y_{i+1}| > K)
\bigr) \\[-2pt]
&&\qquad\quad{} + C_2(g) \lambda_3(h) \sumiNodd
\bigl(
\bb E (|X_i|^3+X_i^2|X_{i+1}| +|X_i|X_{i+1}^2
+ |X_{i+1}|^3;\\[-2pt]
&&\qquad\quad\hspace*{200.3pt}
|X_i|+|X_{i+1}| \le K) \\[-2pt]
&&\qquad\quad\hspace*{87pt}{} + \bb E (|Y_i|^3+Y_i^2|Y_{i+1}|
+|Y_i|Y_{i+1}^2 +
|Y_{i+1}|^3;\\[-2pt]
&&\qquad\quad\hspace*{205.6pt} |Y_i|+|Y_{i+1}| \le K)\bigr),
\end{eqnarray*}
where $C_1(g) = \|g'\|_\infty+ \|g''\|_\infty$ and
$C_2(g) = \|g'\|_\infty+ 3\|g''\|_\infty+\|g'''\|_\infty$.
\end{theorem}

%

We prove Theorem~\ref{ComplexChat2008Thm} in the \hyperref[app]{Appendix}.

Theorem~\ref{ComplexChat2008Thm} requires $h$ to be a real-valued function;
thus we will apply Theorem~\ref{ComplexChat2008Thm} to $\Ree(f)$ and
$\Imm(f)$
separately. We will give the application to $\Ree(f)$ below, noting
that the
same argument applies with $\Imm$ replacing $\Ree$.

Given $g\dvtx\bb R \to\bb R$ a thrice differentiable function, set $U =
\Ree(f({\mathbf X}))$ and $V=\Ree(f(\mathbf{Y}))$, where ${\mathbf X}$
and $\mathbf{Y}$ are
as in the
statement of Lemma~\ref{KrishLem} (notationally, set $N=2n^2$ and define
$X_\ell$ by $X_{1+2n(i-1)+2(j-1) + k}:=\Xijk$).
Note that from the assumption in Lemma~\ref{KrishLem} that the $\Xijk
$ are
sparse\vspace*{2pt} versions of the $\Yijk$, the hypotheses in (\ref{NewChatConds})
are automatically satisfied. Also, the independence hypotheses in
Theorem~\ref{ComplexChat2008Thm} follow from the definitions\vspace*{2pt} of $\Xijk
$ and
$\Yijk$ in Lemma~\ref{KrishLem}. Finally, noting that $\lambda_r
(\Ree f) \le
\lambda_r(f)$, and noting that for our function $f$ we have
\[
\lambda_2(f)= \sup\biggl\{\frac{|v|^{-4}}{n^2}, \frac
{2|v|^{-3}}{n^2}
\biggr\}
\]
and
\[
%
\lambda_3(f)= \sup\biggl\{\frac{|v|^{-6}}{n^3}, \frac
{2|v|^{-5}}{n^{3}}, \frac{6|v
|^{-4}}{n^{5/2}} \biggr\},
\]
we may apply Theorem~\ref{ComplexChat2008Thm} to get
%
%
\begin{eqnarray}
&&
|\bb E g(U) - \bb E g(V)|\nonumber \\
\label{bd2term}
&&\qquad\le
C_1(g) \lambda_2(h)\nonumber\\[-8pt]\\[-8pt]
&&\qquad\quad{}\times\sum_{i=1}^n\sum_{j=1}^n \bigl(
\bb E \bigl(\bigl(X_{i,j}^{(0)}\bigr)^2 + \bigl|X_{i,j}^{(0)}X_{i,j}^{(1)}
\bigr|+\bigl(X_{i,j}^{(1)}\bigr)^2; \bigl|X_{i,j}^{(0)}\bigr|+\bigl|X_{i,j}^{(1)}\bigr| > K\bigr)\nonumber\\
&&\qquad\quad\hspace*{45.6pt}{} + \bb E \bigl(\bigl(Y_{i,j}^{(0)}\bigr)^2 +
\bigl|Y_{i,j}^{(0)}Y_{i,j}^{(1)}\bigr|+\bigl(Y_{i,j}^{(1)}\bigr)^2;
\bigl|Y_{i,j}^{(0)}\bigr|+\bigl|Y_{i,j}^{(1)}\bigr| > K\bigr)
\bigr) \nonumber\\
\label{bd3term}
&&\qquad\quad{} + C_2(g) \lambda_3(h)\\
&&\qquad\quad\hspace*{11pt}{}\times\sum_{i=1}^n\sum_{j=1}^n
\bigl(\bb E \bigl(\bigl|X_{i,j}^{(0)}\bigr|^3
+\bigl(X_{i,j}^{(0)}\bigr)^2
\bigl|X_{i,j}^{(1)}\bigr|
+\bigl|X_{i,j}^{(0)}\bigr|\bigl(X_{i,j}^{(1)}\bigr)^2 +
\bigl|X_{i,j}^{(1)}\bigr|^3;\nonumber\\
&&\qquad\quad\hspace*{203.8pt}\bigl|X_{i,j}^{(0)}\bigr|+\bigl|X_{i,j}^{(1)}\bigr| \le K\bigr)
\nonumber\\
&&\qquad\quad\hspace*{57.2pt}{} + \bb E \bigl(\bigl|Y_{i,j}^{(0)}\bigr|^3+\bigl(Y_{i,j}^{(0)}\bigr)^2
\bigl|Y_{i,j}^{(1)}\bigr| +\bigl|Y_{i,j}^{(0)}\bigr|\bigl(Y_{i,j}^{(1)}\bigr)^2 +
\bigl|Y_{i,j}^{(1)}\bigr|^3;\nonumber\\
&&\qquad\quad\hspace*{211.2pt} \bigl|Y_{i,j}^{(0)}\bigr|+\bigl|Y_{i,j}^{(1)}\bigr|
\le K\bigr)
\bigr).\nonumber
\end{eqnarray}
%
%
Choose $K= \varepsilon\sqrt n$, where $\varepsilon>0$ is a small
positive constant.
The double-sum term in (\ref{bd3term}) is bounded by
$\varepsilon
$ times a
constant depending only on $g$ and $v$ [here, we used that $\bb E(
|X|^3;
|X| \le K) \le K \bb E(X^2)$ for any real random variable~$X$]. Also,\vspace*{1pt}
using the fact that $ab \le\frac12(a^2+b^2)$ for any positive real numbers
$a$ and $b$, the double-sum term in (\ref{bd2term}) is bounded by
another constant depending only on $g$ and $v$ times the quantity
\begin{eqnarray*}
&&\frac{1}{n^2} \sum_{i=1}^n \sum_{j=1}^n
\bb E\bigl(\bigl(X_{i,j}^{(0)}\bigr)^2+\bigl(X_{i,j}^{(1)}\bigr)^2;
\bigl|X_{i,j}^{(0)}\bigr|+\bigl|X_{i,j}^{(1)}\bigr| > \varepsilon\sqrt
n\bigr)\\
&&\qquad\hspace*{0pt}{} +
\bb E\bigl(\bigl(Y_{i,j}^{(0)}\bigr)^2+\bigl(Y_{i,j}^{(1)}\bigr)^2;
\bigl|Y_{i,j}^{(0)}\bigr|+\bigl|Y_{i,j}^{(1)}\bigr| > \varepsilon\sqrt
n\bigr).
\end{eqnarray*}
%
Since the random variables $\Yijk$ do not change with $n$, it is clear from
monotone convergence that $\bb E((Y_{i,j}^{(0)})^2+(Y_{i,j}^{(1)})^2;
|Y_{i,j}^{(0)}|+|Y_{i,j}^{(1)}|
> \varepsilon\sqrt n) \to0$
as $n \to\infty$.
Thus, it is sufficient to show that
$\bb E((\Xijk)^2; |\Xijk| > \varepsilon\sqrt n)
\to0$ as $n\to\infty$. Recall that $X_{i,j}^{(0)} + \ci
X_{i,j}^{(1)}$ is an
i.i.d. copy
of $X\Irho/\sqrt\rho$, where~$X$ is a complex random variable with
mean zero
and variance one, and note that $|X_{i,j}^{(0)}|+|X_{i,j}^{(1)}| >
\varepsilon\sqrt n$
implies that $\sqrt{|X_{i,j}^{(0)}|^2+|X_{i,j}^{(1)}|^2} > \varepsilon
\sqrt
{2n/3}$. We have
that
\begin{eqnarray*}
&&\bb E\biggl( \biggl|\frac{ X\Irho}{\sqrt\rho}\biggr|^2 ; \biggl|\frac{
X\Irho}{\sqrt\rho}\biggr|
> \varepsilon\sqrt{2 n/3}\biggr)\\
&&\qquad\le
\bb E\biggl( \biggl|\frac{ X\Irho}{\sqrt\rho}\biggr|^2 ;
|X| > \varepsilon
\sqrt{2\rho
n/3}\biggr) \\
&&\qquad=
\bb E\bigl( |X|^2 ; |X| > \varepsilon\sqrt{2\rho n/3}\bigr),
\end{eqnarray*}
where the last equality follows by the independence of $\Irho$ and $X$.
Finally, by monotone convergence again, we see that $ \bb E(
|X|^2 ;
|X| > \varepsilon\sqrt{2\rho n/3})\to0$ as $n\to\infty$,
completing the
proof.
\end{pf}


\begin{appendix}\label{app}

\section*{Appendix: A complex version of Chatterjee's invariance~theorem}

In this Appendix, we prove Theorem~\ref{ComplexChat2008Thm}, which is a
version of~\cite{Chatterjee2008}, Theorem~1.1, for the complex
numbers. In
order to prove the result in the complex case, we treat the real and complex
parts of each random variable as separate, possibly dependent real random
variables. The fact that the real and complex parts of a random
variable may
depend on each other introduces some complications. Our approach is modeled
on that in~\cite{Chatterjee2008}, with the main differences being that
we use
the Lindeberg argument on pairs of random variables, rather than on single
random variables, and also we also use two-dimensional Taylor expansions.
%
%
%
\begin{pf*}{Proof of Theorem~\ref{ComplexChat2008Thm}}
Let $\Psi:= g \of h$, which is a function from $\bb R^N \to\bb R$.
Later in
the proof we will apply the two-dimensional version of Taylor's theorem
to~$\Psi$, and so to start we will establish bounds on the partial
derivatives of
$\Psi$. We will use the notation $\partial_{(i_1,i_2,\ldots,i_k)}
\Psi$ as
shorthand for $\frac{\partial^k\Psi}{\partial x_{i_1}\,\partial
x_{i_2} \cdots\,
\partial
x_{i_k}}$. Note that the order of the coordinates $(i_1,i_2,\ldots
,i_k)$ is
unimportant for $1\le k\le3$ since $\Psi$ has continuous partial derivatives
(including mixed partials) by assumptions on $h$ and~$g$.

Note that $\deli\Psi(x) = g'(h(x)) \,\partial_i h(x)$, and so taking further
partial derivatives one can compute that
\[
\delij\Psi(x) = g''(h(x))\,\deli h(x) \,\delj h(x) + g'(h(x)) \,\delij h(x).
\]
Thus, $\mathop{\sup_{x\in\bb R^N}}_{1\le i,j\le N} |\delij
\Psi(x) | \le
C_1(g)\lambda_2(h)$.

Taking further partial derivatives, one can compute that
\begin{eqnarray*}
\partial_{i,j,k} \Psi(x) &=& g'''(h(x))\,\deli h(x) \,\delj h(x)
\,\partial_k h(x)
\\
&&{} +
g''(h(x)) \,\partial_{ik} h(x) \,\delj h(x) + g''(h(x)) \,\deli h(x)\,
\partial_{jk}
h(x) \\
&&{} + g''(h(x)) \,\partial_k h(x) \,\delij h(x) + g'(h(x)) \,\partial
_{ijk} h(x).
\end{eqnarray*}
Thus, $\mathop{\sup_{x \in\bb R^N}}_{1\le i,j,k\le N} \le
C_2(g)\lambda_3(h)$.\vspace*{2pt}

For $1\le i \le N$ and $i$ odd, define
\begin{eqnarray*}
Z_i &:=& ( X_1,\ldots,X_{i-1}, X_i,X_{i+1}, Y_{i+2},\ldots, Y_N),\\
W_i &:=& ( X_1,\ldots,X_{i-1}, 0 , 0 , Y_{i+2},\ldots, Y_N)
\end{eqnarray*}
with $Z_{-1}:=( Y_1,\ldots,Y_N)$ and $Z_{N-1}:=( X_1,\ldots,X_N)$.
Also, for $1\le i \le N$ and~$i$ odd, define
\begin{eqnarray*}
R_i &:=& \Psi(Z_i) -\Psi(W_i) - X_i \,\deli\Psi(W_i) - X_{i+1}\,\partial_{i+1}
\Psi(W_i) \\
&&{} - \frac{X_i^2}{2} \,\delii\Psi(W_i) - \frac{X_{i+1}^2}{2}
\,\partial_{i+1,i+1} \Psi(W_i) \\
&&{}- X_i X_{i+1} \,\partial_{i,i+1}
\Psi(W_i),
\\
T_i &:=&\Psi(Z_{i-2}) -\Psi(W_i) - Y_i \,\deli\Psi(W_i) -
Y_{i+1}\,\partial_{i+1}
\Psi(W_i) \\
&&{} - \frac{Y_i^2}{2} \,\delii\Psi(W_i) - \frac{Y_{i+1}^2}{2}\,
\partial_{i+1,i+1} \Psi(W_i)\\
&&{} - Y_i Y_{i+1} \,\partial_{i,i+1} \Psi(W_i).
\end{eqnarray*}

Note that by Taylor's theorem in two dimensions and bounds on the
partials of~$\Psi$, we have for odd $i$ that
\begin{eqnarray*}
|R_i| & \le & C_1(g) \lambda_2(h) ( X_i^2 + |X_i
X_{i+1}| +
X_{i+1}^2), \\
|R_i| & \le & C_1(g) \lambda_2(h) ( Y_i^2 + |Y_i
Y_{i+1}| + Y_{i+1}^2)
\end{eqnarray*}
using second order bounds, and that
\begin{eqnarray*}
|R_i| & \le & C_2(g) \lambda_3(h) ( |X_i|^3 +
X_i^2 |X_{i+1}| +
|X_i| X_{i+1}^2 + |X_{i+1}|^3), \\
|T_i| & \le & C_2(g) \lambda_3(h) ( |Y_i|^3 +
Y_i^2 |Y_{i+1}| +
|Y_i| Y_{i+1}^2 + |Y_{i+1}|^3 )
\end{eqnarray*}
using third order bounds.

We now make use of the Lindeberg principle, writing $|\bb E(g(U))
- \bb E(g(V))|$ in terms of a telescoping sum involving $Z_i$.
\begin{eqnarray*}
&&
| \bb E(g(U)) - \bb E(g(V))| \\[-1pt]
&&\qquad= \Biggl|\sumiNodd\bb E\bigl(\Psi
(Z_i) - \Psi(Z_{i-2}) \bigr) \Biggr| \\[-1pt]
&&\qquad= \Biggl|
\sumiNodd\bb E
\biggl(
\Psi(W_i) + X_i \,\deli\Psi(W_i) + X_{i+1}\,\partial_{i+1}
\Psi(W_i)\\[-1pt]
&&\qquad\quad\hspace*{35.7pt}{}
+ \frac{X_i^2}{2} \,\delii\Psi(W_i)
+ \frac{X_{i+1}^2}{2}\,
\partial_{i+1,i+1} \Psi(W_i)\\[-1pt]
&&\qquad\quad\hspace*{92pt}{} + X_i X_{i+1} \,\partial_{i,i+1} \Psi
(W_i)+R_i\biggr)
\\[-1pt]
&&\qquad\quad{} -
\bb E\biggl(
\Psi(W_i) + Y_i \,\deli\Psi(W_i) + Y_{i+1}\,\partial_{i+1}
\Psi(W_i)\\[-1pt]
&&\qquad\quad\hspace*{25.4pt}{}+ \frac{Y_i^2}{2} \,\delii\Psi(W_i)
+ \frac{Y_{i+1}^2}{2}\,
\partial_{i+1,i+1} \Psi(W_i)\\[-1pt]
&&\qquad\quad\hspace*{83.4pt}{} + Y_i Y_{i+1} \,\partial_{i,i+1} \Psi
(W_i)+T_i\biggr)
\Biggr|\\[-1pt]
&&\qquad= \Biggl|\sumiNodd\bb E(R_i) - \bb E(T_i) \Biggr|.
\end{eqnarray*}
Note that in the above there is lots of cancellation, for example,
\[
\bb E X_i
X_{i+1} \,\partial_{i,i+1} \Psi(W_i)- \bb EY_i Y_{i+1} \,\partial
_{i,i+1} \Psi(W_i)=0
\]
by the independence assumptions along with the assumption that $\bb E X_i
X_{i+1}=\bb EY_i Y_{i+1} $.

To complete the proof we bound $R_i$ and $T_i$ using second order
bounds when
they are small and using third order bounds when they are large,
arriving at
\begin{eqnarray*}
\hspace*{-6pt}&&\Biggl| \sumiNodd\bb E(R_i) - \bb E(T_i) \Biggr|\\[-1pt]
\hspace*{-6pt}&&\qquad\le
C_1(g) \lambda_2(h) \sumiNodd
\bigl(\bb E (X_i^2 + |X_iX_{i+1}|+X_{i+1}^2; |X_i|+|X_{i+1}| >
K) \\[-1pt]
\hspace*{-6pt}&&\qquad\quad\hspace*{74.4pt}{} + \bb E (Y_i^2 + |Y_iY_{i+1}|+Y_{i+1}^2;
|Y_i|+|Y_{i+1}| > K)
\bigr) \\[-1pt]
\hspace*{-6pt}&&\qquad\quad{}+ C_2(g) \lambda_3(h) 
\sumiNodd
\bigl(
\bb E (|X_i|^3+X_i^2|X_{i+1}| +|X_i|X_{i+1}^2
+ |X_{i+1}|^3;\\[-1pt]
\hspace*{-6pt}&&\qquad\quad\hspace*{200.7pt}|X_i|+|X_{i+1}| \le K) \\[-1pt]
\hspace*{-6pt}&&\qquad\quad\hspace*{86.4pt}{} + \bb E (|Y_i|^3+Y_i^2|Y_{i+1}|
+|Y_i|Y_{i+1}^2 +
|Y_{i+1}|^3;\\[-1pt]
\hspace*{-6pt}&&\qquad\quad\hspace*{206.3pt} |Y_i|+|Y_{i+1}| \le K)
\bigr).\qquad
\end{eqnarray*}
\upqed\end{pf*}
\end{appendix}

\section*{Acknowledgments}

I would like to thank Van Vu for suggesting the project in the first
place and for support throughout, and I would also like to thank
Terence Tao for many useful discussions and ideas. I am also grateful
to Persi Diaconis for advice and discussions about this project.
Finally, thanks is due to Tim Rogers for useful comments in the
revision stage and also to the anonymous referee for insightful
suggestions and a number of useful corrections that have improved the
paper.

\printaddresses

\end{document}